\documentclass[10pt]{article}

\usepackage[marginparwidth=35pt,textheight=23cm,voffset=-1.2cm,footskip=1.3cm,voffset=0.2cm]{geometry}

\usepackage{amsmath}

\usepackage{amssymb}

\usepackage{hyperref}

\usepackage{comment}
\usepackage{enumerate}

\usepackage{xcolor}

\usepackage[activate={true,nocompatibility},final,tracking=true,kerning=true,factor=1100,stretch=10,shrink=10]{microtype}

\usepackage{refcheck}

\usepackage[cmtip,all]{xy}

\usepackage{mathrsfs}
\usepackage{amsthm}

\usepackage[T1]{fontenc}
\usepackage{lmodern}

\newtheorem{thm}{Theorem}[section]

\newtheorem{defn}[thm]{Definition}
\newtheorem{prop}[thm]{Proposition}
\newtheorem{lemma}[thm]{Lemma}
\newtheorem{cor}[thm]{Corollary}

\newtheoremstyle{boldremark}
    {\dimexpr\topsep/2\relax} 
    {\dimexpr\topsep/2\relax} 
  {}          
    {}          
    {\bfseries} 
    {.}         
    {.5em}      
    {}          

\theoremstyle{boldremark}
\newtheorem{remark}[thm]{Remark}

\newcommand{\N}{{\ensuremath{\mathbb{N}}}}

\newcommand{\R}{{\ensuremath{\mathbb{R}}}}

\newcommand{\longsquiggly}{\xymatrix@C=1.5em{{}\ar@{~>}[r]&{}}}

\newcommand{\spt}{{\ensuremath{\textrm{spt}}}}

\newcommand{\wassspaceonM}{{\ensuremath{\mathscr{P}(M)}}}

\newcommand{\wassspaceonGammaK}{{\ensuremath{\mathscr{P}(\Gamma^K)}}}

\newcommand{\wassspaceonOmega}{{\ensuremath{\mathscr{P}(\bar{\Omega})}}}

\newcommand{\wassspaceonGammamKzero}{{\ensuremath{\mathscr{P}_{m_0}(\Gamma^K)}}}

\newcommand{\s}[1]{{\mathcal #1}}

\newcommand{\bb}[1]{{\mathbb #1}}

\usepackage{mathtools}

\newcounter{step}
\newcounter{case}
\usepackage[backend=biber,style=alphabetic]{biblatex}

\usepackage{commath}

\usepackage{pdfsync}
\usepackage{doi}

\title{A note on deterministic mean field games of controls 
with state constraints: existence of mild solutions}
\author{Jameson Graber\footnote{Department of Mathematics, 
Baylor University. \href{mailto:Jameson_Graber@baylor.edu}{Jameson\_Graber@baylor.edu}} 
\and Sergio Mayorga\footnote{Innopolis University.
\href{mailto:me@mayorga.ru}{me@mayorga.ru}}}
\begin{document}

\maketitle 

\begin{abstract}
We show the existence of \emph{mild solutions} for a first-order 
mean field game of controls under the state constraint that trajectories 
be confined in a closed and bounded set in euclidean space. 
This extends the results of \cite{cannarsacapuani1} to the case of 
a mean field game of controls. Our controls are velocities and we find 
that the existence of an equilibrium is complicated by the requirement 
that they should have enough regularity. We solve this by imposing a 
small Lipschitz constant on the dependence of the Lagrangian on the 
joint measure of states and controls.
\end{abstract}

\tableofcontents

\vspace{0.5cm}




\section{Introduction}

Mean field games have been introduced to study large-scale 
interactions between rational agents \cite{lionsjapanese,huang2006large}. 
In a mean field game, one assumes a population of anonymous 
identical players.
A representative player solves an optimal control problem 
depending on a flow of measures $(\nu_t)_{0 \leq t \leq T}$, e.g.
\begin{equation}\label{eq:optimal control}
	\inf_{x(\cdot)} \cbr{\int_0^T l(t,x(t),\dot x(t),\nu_t)\dif t + l_T(x(T),\nu_T)}
\end{equation}
for a given running cost $l$ and final cost $l_T$.
Here and throughout this article the game is deterministic, 
i.e.~the players' dynamics are not subject to noise.
In most of the literature on mean field games, $\nu_t = m_t$ 
is a measure only on the state space, i.~e.~the set where
$x(t)$ take values. Also, most of the works do not impose 
constraints on this state space. In this paper we run 
counter these features by considering measures jointly
on state space and control space, and state constraints
simultaneously.

In the case of measures on the state space ($\nu_t=m_t$),
a \emph{Nash equilibrium} occurs when $m_t$ 
is precisely the distribution of states given that every 
player follows optimal trajectories of \eqref{eq:optimal control}.
One then derives the following system of partial 
differential equations for the equilibrium:
\begin{equation}\label{eq:classic MFG}
	\begin{cases}
		-\partial_t u + h(t,x,\nabla u,m_t) = 0, & u(x,T) = l_T(x,m_T),\\
		\partial_t m_t - 
		\nabla \cdot (m_t \nabla_p h(t,x,\nabla u,m_t)) = 0, 
		& m_t|_{t = 0} = m_0,
	\end{cases}
\end{equation} 
where, in this introduction, $h(t,x,p,m) := \sup_{v}\{-p\cdot v - l(t,x,v,m)\}$.
The first equation in System \eqref{eq:classic MFG} is 
a backward-in-time Hamilton-Jacobi equation with terminal 
condition $l_T$ that depends on $m_T$, and the second is 
a forward-in-time continuity equation with given initial 
condition $m_0$.
System \eqref{eq:classic MFG} is now commonplace in mean 
field game theory, especially in the case where $h$ can be
written in the form $h(t,x,p,m) = H(x,p) - F(x,m)$.
See \cite{cardaliaguetnotes,cardaliaguetgraber,cardaliaguet2015weak,grabermeszaros} 
for results on the well-posedness of such systems.
Mean field games arise as formal limits, as $N\to\infty$, of games with $N$ players. For the model considered
in this paper, the game in question is the following: player $i\in\{1,\ldots,N\}$ starts at a point of $\bar{\Omega}$,
chooses a trajectory $x_i(\cdot)$ subject to the state constraint $x_i(t)\in\bar{\Omega}$ for all $t\in[0,T]$, and
seeks to minimize
\begin{equation*}
	\int_0^T l(t,x_i(t),\dot x_i(t),\nu^N_t)\dif t + l_T(x_i(T),\nu^N_T), \qquad
	\nu^N_t := \frac{1}{N}\sum_{j=1}^N\delta_{(x_j(t),\dot{x}_j(t))},
\end{equation*}
so that the players interact through the empirical joint distribution of their states and velocities.
We stress that in this paper the connection with this game is only formal: we prove neither that the Nash equilibria
of the $N$-player game converge to the mean field equilibria constructed below, nor that the latter yield approximate
Nash equilibria of the $N$-player game; these questions are beyond the scope of the present paper.
Results of this kind are by now well documented in many settings without state constraints, see
e.g.~\cite{huang2006large,cardaliaguetnotes,carmonabible,lackerlimit} for mean field games and
\cite{djete2023,laurieretangpi2022} for mean field games of controls with stochastic dynamics.

A natural extension is to consider $\nu_t$ a measure on the 
Cartesian product of the state space with the space of controls.
In this article we will take controls to be velocities.
Then a Nash equilibrium occurs when $\nu_t$ is precisely 
the distribution of all possible values $(x(t),\dot x(t))$ 
for optimal trajectories of \eqref{eq:optimal control}.
In this case we have a \emph{mean field game of controls} 
\cite{cardaliaguettradecrowding} (also called 
\emph{extended mean field game} 
\cite{gomes2014extended,gomesvoskanyanextended}).
The corresponding system of partial differential 
equations now reads
\begin{equation}
\label{eq:MFGC}
	\begin{cases}
		-\partial_t u + h(t,x,\nabla u,\nu_t) = 0, 
		& u(x,T) = l_T(x,\nu_T),\\
		\partial_t m_t - 
		\nabla \cdot (m_t \nabla_p h(t,x,\nabla u,\nu_t)) = 0, 
		& m_t|_{t = 0} = m_0,\\
		\del{I,-\nabla_p h(t,\cdot,\nabla u,\nu_t)}_{\#} m_t 
		= \nu_t,
	\end{cases}
\end{equation}
where $I$ is the identity mapping and $\#$ is the push-forward 
operator; thus the third equation in \eqref{eq:MFGC} means that 
$\nu_t$ is the push-forward of $m_t$ through the map 
$x \mapsto (x,-\nabla_p h(t,x,\nabla u(x,t),\nu_t))$.
Cardaliaguet and Lehalle give an abstract result on the 
well-posedness of System \eqref{eq:MFGC} over 
$\bb{R}^d \times [0,T]$ in \cite{cardaliaguettradecrowding} 
(see also 
\cite{gomes2014extended,gomesvoskanyanextended,grabermullenixweak}). 

As said earlier, in this paper we consider mean field games of controls with 
state constraints, i.e.~where the trajectory $x(t)$ must 
remain within some given closed subset of $\bb{R}^d$. 
For the more common case where $\nu_t = m_t$ and 
$h(t,x,p,m) = H(x,p) - F(x,m)$ (i.e.~the case of a 
``separable Hamiltonian'', which we don't assume here),
a full analysis of first-order 
mean field games with state constraints has been completed by 
Cannarsa, Capuani, and Cardaliaguet 
\cite{cannarsacapuani1,cannarsaconstcalc,cannarsacapuani2}.
Their program has three steps.
The first is to prove the existence of ``mild solutions'' to 
the mean field game with state constraints, which is motivated 
in the following way.
For sufficiently smooth Hamiltonians and terminal conditions, 
viscosity solutions of the Hamilton-Jacobi equation without 
state constraints automatically possess $\s{C}^{1,1}$ regularity.
This means that the optimal feedback control 
$-\nabla_p H(x,\nabla u)$ is 
defined on a sufficiently ``large'' set so that the continuity 
equation is well-posed.
In the presence of state constraints, it is known that such 
regularity can fail for Hamilton-Jacobi equations, 
e.g.~\cite[Section 1.1]{cardaliaguet2006regularity}.
Thus, a different approach to the existence of 
Nash equilibria is 
required.
Start with the Lagrangian formulation of the problem, which is 
taken from \cite{benamousantambrogio}, rewriting the optimal control 
problem \eqref{eq:optimal control} as depending on a measure $\eta$ 
on the space of possible paths 
(see formula \eqref{eq:switchtoeta} below). By definition, a \emph{mild solution} 
occurs when $\eta$-a.e.~trajectory is a minimizing trajectory; such 
an $\eta$ is called a \emph{constrained equilibrium}. A multivalued 
mapping $\eta\mapsto E(\eta)$ is associated with this definition, 
to which Cannarsa and Capuani apply the Kakutani fixed point 
theorem \cite{cannarsacapuani1}, thereby showing the existence 
of mild solutions.
Once this first step is completed, the next step is to study 
regularity properties of optimizers to control problems with 
state constraints \cite{cannarsaconstcalc} and apply the 
results to the coupled system \cite{cannarsacapuani2}.

Our goal is to follow this same program for mean field games of 
controls with state constraints.
For the present study we will address only the first step, 
namely the existence of appropriately defined mild solutions
(see Definition \ref{defn:mildsolution}).
In doing so, we encounter some difficulties that did not appear 
in \cite{cannarsacapuani1}.
We need to come up with an appropriate definition of constrained 
equilibrium that involves both states and controls, and show 
that the analogue of the mapping $E$ can be defined from a
 compact set into itself, and that it is closed.
For compactness, since we are considering measures on both 
states and velocities, we require a second-order estimate on 
optimal trajectories; this is in contrast to \cite{cannarsacapuani1}, 
in which only first-order estimates were needed in the first step.
Fortunately, we are able to apply results from \cite{cannarsaconstcalc} 
to obtain the desired compactness.
To show that $E$ is closed, we approximate constrained trajectories by projecting onto
$\bar{\Omega}$, as in \cite[Section 3.1]{cannarsacapuani1}; this is
possible because, under our small data condition, the optimal controls in
our class of regular controls are optimal among all admissible controls.

In \cite{bonnansgianatti}, the authors have also studied a 
class of mean field games of controls with mixed control-state 
constraints, using a Lagrangian approach as we do.
In their model, the distributions of controls enters into the 
cost to players exclusively via a price function, which 
depends on the average of controls.
They assume a qualification assumption (Hypothesis (H5)) 
on the controls that allows them to prove the regularity 
of minimizers to the control problem, followed by 
well-posedness of the price function as a fixed point.
By contrast, we will assume no restriction on the controls, 
but only on the states.
Regularity of optimizers will come from the results of 
\cite{cannarsaconstcalc} on Euler-Lagrange equations for 
constrained optimal control.
Moreover, we assume a more abstract dependence on the 
distribution of controls, where the Lagrangian is Lipschitz
 in the variable $\nu$ with respect to a Wasserstein metric; 
 the trade-off is that we make a smallness assumption 
 on the Lipschitz constant (see Section \ref{section:smalldata}).
 In extended mean‑field games without state constraints, Gomes and Voskanyan \cite{gomesvoskanyanextended}
 prove well‑posedness under a solvability condition that requires the dependence of the Hamiltonian on the measure to be sufficiently Lipschitz. This condition guarantees that
 a certain fixed‑point mapping from a space
 of random variables to itself is a contraction. In our state‑constrained setting, the smallness of 
the Lipschitz constant of 
$l$ and $D_v l$ with respect to the measure plays an analogous role (see Remark \ref{remark:smallnesscondition}
below), but the analysis is more demanding because the boundary forces us to control second‑order derivatives of
the trajectories. As a result, the smallness condition enters in a stronger, nonlinear way (see the cubic growth of the function $\Phi_2$ in Lemma
\ref{lemma:Ksub1}). Whether the cubic smallness can be relaxed to a linear one, or replaced by a more structural solvability condition as in the unconstrained case, is an open question.

The motivation for our study is far from purely mathematical.
Indeed, many applications of mean field game theory in 
economics require mean field games of controls with state 
constraints \cite{achdouetalmacro,achdou2017income,chan2017fracking}.
We consider the present article an important step toward 
rigorously establishing the well-posedness of these models. 
Let us now present a brief overview of the main points in this article.

\paragraph{AI use disclosure} 
The first version of this article was written entirely without AI.
We used Claude Fable 5.2 and DeepSeek V4.1
for the revision of the manuscript to help fix gaps found by the referees.
None of the gaps
affected the validity of our work's results.
They mainly concerned the proofs of Lemmata \ref{lemma:closedandconvexvalues}
and \ref{lemma:Ksub1} and some of the fine technicalities 
in the adaptation (Proposition \ref{prop:adjointarc}) 
of Theorem \ref{thm:cardcannarsa} from
\cite[Thm.~3.1]{cannarsaconstcalc}, which in the pre-referee
version were confined to an appendix.

We manually checked and edited the AI output before 
incorporating it into our work, and we take full responsibility
for the contents of the paper.

\subsection*{Main results}

Let $\Omega\subset\R^d$ be a bounded domain with boundary of class $\s{C}^3$, and
let $T>0$, $X\subset \R^d$ a fixed open bounded set including $\bar{\Omega}$.
We assume that the running cost $l$ and the terminal cost $l_T$ satisfy hypotheses
(L-\ref{L-i}), (L-\ref{L-ii}), (L-\ref{L-iii}), (T-\ref{T-i}) and (T-\ref{T-ii}) of Section \ref{subsection:notationandhypotheses};
in particular, $l$ is uniformly convex in the velocity, and $l$, $D_vl$ and $l_T$ are Lipschitz continuous
in the measure variable with constant $c_1$.
Our main result is the following.
\begin{quote}
\textit{There exists a constant $\bar{c}>0$, which depends only on $T$, $\Omega$, $X$ and the constants in the hypotheses other than $c_1$,
such that, if $c_1\leq\bar{c}$, then for every initial distribution of states $m_0\in\mathscr{P}(\bar{\Omega})$ there exists
a mild solution of the mean field game of controls with state constraints
(Definition \ref{defn:mildsolution} and Theorem \ref{thm:mildsolution}).}
\end{quote}
We also show that, if $l$ and $l_T$ are strictly monotone in the measure variable,
then all mild solutions have the same value function (Theorem \ref{thm:uniqueness}).
The proof of the existence theorem is organized as follows. In Section \ref{section:preliminaries} we state the hypotheses
and recall a regularity result of \cite{cannarsaconstcalc} (Theorem \ref{thm:cardcannarsa}). In Section \ref{section:smalldata}
we prove the estimates on the minimizers (Proposition \ref{prop:adjointarc} and Lemma \ref{lemma:Ksub1}) and derive from them the
small data condition (Corollary \ref{cor:ifmonesmallenough}). In Section \ref{section:constrainedequilibrium} we define
the constrained equilibria, show that they exist by a fixed-point argument
(Lemmas \ref{lemma:optimaloverAC}, \ref{lemma:Gammacetaisclosed}, \ref{lemma:closedandconvexvalues} and \ref{lemma:existenceofCEMFGC}),
and prove the two theorems. We now describe these steps in more detail.

In section \ref{section:preliminaries} we fix some terminology 
and borrow a result of \cite{cannarsaconstcalc} on the regularity 
of the minimizers of the constrained problem for a cost functional 
with time-dependent Lagrangian $l$ (Theorem
\ref{thm:cardcannarsa}). All the trajectories are to remain in a closed, 
bounded smooth domain $\bar{\Omega}$,  
the space of states.  Let 
$\Gamma=\{\gamma(\cdot)$ 
$
= 
(x(\cdot),u(\cdot)) \ | \ x(\cdot)\in\Gamma_1,
u(\cdot)=\dot{x}(\cdot)\in \Gamma_2\}$
be a
set of ``pairs'' of trajectories and controls,
on the time interval $[0,T]$; see (\ref{eq:GammaK}) below
for full definiteness. Suppose $\eta$ is a probability measure 
on the set $\Gamma$,  and for $\gamma\in \Gamma$, 
$e_t(\gamma)=\gamma(t)=(x(t),u(t))$,  with the 
additional property 
that the first marginal of $(e_0)_{\#}\eta$ equals a 
measure $m_0$ on $\bar{\Omega}$ fixed beforehand. 
For any given $x_0\in\bar{\Omega}$, consider the functional
\begin{align*}
I^{x_0}[u(\cdot);\eta] := 
\int_0^T l (t,x_{x_0}^u(t),u(t),(e_t)_{\#}\eta) dt +
l_T(x_{x_0}^u(T),(e_T)_{\#}\eta)
\end{align*}
where $x_{x_0}^u(\cdot)$ is the trajectory starting at $x_0$ 
with velocity $u(\cdot)$. This functional can be viewed as 
acting on $\Gamma_2$, the set of second components of $\Gamma$. 
Roughly speaking, for a hopeful definition of equilibrium, 
we would like the sets of minimizers of this functional 
in $\Gamma$ to be in the same class as the (second components of the) 
paths $\gamma$ that ``make up'' the input measure $\eta$ on $\Gamma$. 
This
is achieved by Corollary \ref{cor:ifmonesmallenough} below. 
It comes at the expense of requiring the Lipschitz constant of $l$
and $D_v l$ with respect to its last variable (the measure variable) 
\emph{sufficiently small}, so at this moment we warn the reader upfront
that the result in this paper is valid under this ``small data'' 
condition in regards of the measure variable; this restriction
does not appear in the case with no state constraints 
\cite{cardaliaguettradecrowding}. Having guaranteed that data 
$\eta$ in the space of probability measures on the suitable set 
$\Gamma^K$ produces minimizers with second components in $\Gamma^K_2$,  
we proceed to declare $\eta$ a 
\textit{constrained equilibrium of mean field game of controls} 
(``CEMFGC'', Definition \ref{defn:CEMFGC} below, inspired 
directly by \cite{cannarsacapuani1}), if 
$\eta$-a.e.~$\gamma=(x(\cdot),u(\cdot))\in\Gamma^K$ has 
the property that $u(\cdot)$ is an optimal control for 
$I^{x(0)}[ \cdot ; \eta]$;
the key here is, of course, the ``$\eta$-a.e.''~part. A set-valued mapping 
$E:\mathscr{P}_{m_0}(\Gamma^K)\leadsto \mathscr{P}_{m_0}(\Gamma^K)$ 
is \footnote{The symbols \(\mathscr{P}(\ldots)\) are used to denote 
spaces of probability measures; see Section \ref{subsection:notationandhypotheses} below.
For 
the definition of \(\mathscr{P}_{m_0}(\Gamma^K)\) see
Section \ref{section:constrainedequilibrium} further below.}
associated with this definition (see Remark \ref{remark:verifyequivalence}),
 and the main task becomes to show that $E$ has a fixed point. 
 Now, to prove that $E$ does not take on empty values requires 
 some work, since the multi-set mapping $(\eta,x_0)\mapsto\Gamma^{\eta}[x_0]$, 
 with $\Gamma^{\eta}[x_0]$ the set of optimal controls in $\Gamma_2^K$ 
 for the functional $I^{x_0}[\cdot ; \eta]$ should be closed, which 
 is shown in Lemma \ref{lemma:Gammacetaisclosed}, and it is there
 that we use the approximation of constrained trajectories (Lemma
 \ref{lemma:approximation}),
 for the following reason, in somewhat vague terms: 
 one wishes to show that a control $u(\cdot)$ with initial state 
 $x_0$ is optimal, provided there is a sequence $x_n\to x_0$ with 
 a corresponding sequence of optimal $u_n(\cdot)$ with initial states 
 $x_n$, and such that the controls $u_n(\cdot)$ converge to $u(\cdot)$. 
 To compare the cost of the control $u(\cdot)$ with that of some other, 
 arbitrary control $\tilde{u}(\cdot)$ (corresponding to the same initial state $x_0$), 
 one would like to approximate $\tilde{u}(\cdot)$ by a sequence $\tilde{u}^n(\cdot)$ 
 with initial states $x_n$,  use the fact that the cost of $u_n(\cdot)$ 
 is less than the cost of $\tilde{u}^n(\cdot)$,  and then apply the continuity 
 of the functional to pass to the limit and conclude that the 
 cost of $u(\cdot)$ is less than the cost of $\tilde{u}(\cdot)$. 
 This procedure is standard, and the approximating sequence
 $\{\tilde{u}^n(\cdot)\}_{1}^{\infty}$ is constructed as in \cite{cannarsacapuani1}, by projecting
 onto $\bar{\Omega}$ a translation of the trajectory. The velocities obtained in this way may have discontinuities; we can
 nevertheless use them because, by Corollary \ref{cor:ifmonesmallenough}, the controls $u_n(\cdot)$ are
 optimal not only among the regular controls, but among all admissible ones (Lemma \ref{lemma:optimaloverAC}). 

The existence of a CEMFGC then follows by the Kakutani-Fan-Glicksberg 
theorem; see Lemmas \ref{lemma:closedandconvexvalues} 
and \ref{lemma:existenceofCEMFGC}. This justifies the 
notion of \emph{mild solution} (Definition \ref{defn:mildsolution} 
and Theorem \ref{thm:mildsolution}), essentially just saying that we 
have a value function $V(t,x)$ given as the solution of an optimal 
control problem that is ``in equilibrium'' with its own data $\eta$,
meaning that the distribution of the set of its optimal pairs $(x(\cdot),u(\cdot))$ 
 is precisely the support of $\eta$,  the distribution of all trajectories
  and control histories. Theorem \ref{thm:mildsolution} is the main result of this paper;
  Theorem \ref{thm:cardcannarsa}, on the other hand, is a result of
  \cite{cannarsaconstcalc},
  and Section \ref{section:smalldata} contains the estimates that
  support the main result.

\subsection*{Contrast with previous works}

To justify the novelty of our contribution,
let us clarify a bit and expand on some of the points we touched upon above.
The construction of a mean-field game equilibrium is, at its core,
a fixed-point construction. Here, as in \cite{cannarsacapuani1}, we
hold the Lagrangian point of view and are not concerned with 
the differential equations associated with it. 
The task in \eqref{eq:optimal control} is faced by
the agent starting at position \(x_0\) in the
closed bounded domain \(\bar{\Omega}\). If 
the measures \(\nu_t\) are to represent the 
distribution of the players, each of which is solving 
the same task (starting at a different point), then
\(\nu_t\) cannot be arbitrary: it must coincide with 
the distribution ``generated'' by all the trajectories
\(x(\cdot)\) at time \(t\). 

To turn this into a fixed-point story,
one would like to say that given a path of measures
\(\nu_{\cdot}\), with \(\nu_t\in\mathscr{P}(\bar{\Omega})\) for
every \(t\in[0,T]\),
a mapping from, say, \(AC([0,T];\mathscr{P}(\bar{\Omega}))\) into itself 
is created by 
considering the flow of distributions 
\(\tilde{\nu}_{\cdot}\) formed by all the solutions
to \eqref{eq:optimal control} obtained by starting at all 
different \(x_0\in\bar{\Omega}\). However, it is not clear how to
construct these \(\tilde{\nu}_{\cdot}\). Thus, the idea of relaxation
is applied. Attention is switched to the set of paths 
(enjoying an appropriate regularity property, e.g.~the space 
of \(\bar{\Omega}\)-valued
absolutely continuous functions on \([0,T]\)) 
endowed with an appropriate
metric; denote this metric space by \(\Gamma\) and 
let \(\tilde{\Gamma}\) be some subspace of \(\Gamma\) 
(which
we don't specify at this moment, but is chosen 
so as to ensure compactness arguments).
The objective for the generic agent, \eqref{eq:optimal control}, can
be written as the minimization of \(I^{x_0}[x(\cdot);\eta]\) over
all \(x(\cdot)\) in \(\tilde{\Gamma}\) with common starting point \(x_0\):
\begin{equation}\label{eq:switchtoeta}
	I^{x_0}[x(\cdot);\eta] := 
	\int_0^T l (t,x(t),\dot{x}(t),(e_t)_{\#}\eta) dt + 
	l_T(x(T),(e_T)_{\#}\eta)
\end{equation}
where \(e_t(\gamma):=\gamma(t)\), \(\gamma\in \Gamma\), and 
\(\eta\in \mathscr{P}(\tilde{\Gamma})\). Now \(\nu_t\) has been replaced by 
\((e_t)_{\#}\eta\), where \(\eta\) is a distribution of paths. 
Also, it is reasonable to assume that the agents are initially 
distributed according to some \(m_0\in\mathscr{P}(\bar{\Omega})\),
which is to be understood as the distribution of all the \(x(0)\).
This and the introduction of \(\eta\)
allows one to consider the following type of question. Suppose we 
solve \eqref{eq:switchtoeta} for every \(x_0\in \bar{\Omega}\),
and take the element \(\tilde{\eta}\in\mathscr{P}(\tilde{\Gamma})\),
satisfying \((e_0)_{\#}\tilde{\eta} = m_0\),
that is made up of the minimizers 
of \eqref{eq:switchtoeta}, more precisely, the one whose
support consists of minimizers of \eqref{eq:switchtoeta}.
Is \(\tilde{\eta}=\eta\)?
If it is, then the input \((e_t)_{\#}\eta\) in \eqref{eq:switchtoeta}
coincides with the output, i.e.~the resulting distribution of players
at different times \((e_t)_{\#}\tilde{\eta}\). This can in fact be 
regarded as defining a mean field game equilibrium.

As said earlier, for this to work, one must guarantee that the minimizers of 
\eqref{eq:switchtoeta} lie in \(\tilde{\Gamma}\), so that the
construction just described indeed defines 
a mapping from \(\mathscr{P}(\tilde{\Gamma})\) into itself.
Moreover, \(\tilde{\Gamma}\) must be suitable for the application
of Kakutani's fixed point theorem. A priori, one does not know 
whether, given \(\eta\in \mathscr{P}(\tilde{\Gamma})\), the minimizer 
of \eqref{eq:switchtoeta} lies in \(\tilde{\Gamma}\) and not in 
\(\Gamma\setminus\tilde{\Gamma}\). Thus, an appropriate 
subspace \(\tilde{\Gamma}\)
must be identified. In
\cite{cannarsacapuani1},
\(\Gamma = AC([0,T];\bar{\Omega})\) 
and 
\(
	\tilde{\Gamma} 
= \{\gamma \in \Gamma \ | \ \|\dot{\gamma}\|_{L^2([0,T])} \leq K\}
\)
for a constant \(K\) that depends only on the properties of \(l\) and \(l_T\).

\textit{In the case when a dependence in \(l\) and \(l_T\) is introduced that
takes into account not just the distribution of players, but
the distribution of controls} (which, in the simplest case and the one
considered in this work, are simply the velocities), the probability measure
\(\eta\) is, naturally, a measure on a space of pairs: roughly speaking,
the first component of a pair is a trajectory and the second component
is the control that generates it (see \eqref{eq:GammaK}). The problem of identifying
the appropriate subspace \(\tilde{\Gamma}\) is now far more 
delicate.  Consider that now, 
any constraint involving the derivative of \(\gamma\) is 
actually a constraint on the derivative of the trajectory and 
the derivative thereof, which means, a \emph{constraint 
also on the second derivative} of the trajectory.
Given that the domain \(\Omega\) is bounded, this imposes
a difficulty in that to prove that a value of the  constraining bound (analogous to \(K\) above) in fact exists
such that the output paths (i.e.~the minimizers of \eqref{eq:functionalu} below, analogous 
to \eqref{eq:switchtoeta}, but now involving both
the trajectories and the controls) 
lie in \(\tilde{\Gamma}\) when the input \(\eta\) is taken
to be a probability measure on \(\tilde{\Gamma}\). 
To this end, if \(\eta\) is fixed, one can 
treat the fourth argument in \(l\) under the integral in 
\eqref{eq:switchtoeta} as being absorbed by the explicit 
time dependence of \(l\) on \(t\), and this incorportation
now carries \(K\) along with it, while the coefficients 
\(\Lambda\) and \(\beta\)
in the equations of Theorem
\ref{thm:cardcannarsa} of \cite{cannarsaconstcalc} (which are satisfied by 
the minimizing trajectories) are very sensitive to the 
bounds on the time dependence of \(l\). As a result, in general,
there seems to be no hope of finding the
set \(\tilde{\Gamma}\) such 
that we don't exit this set in the minimization problem. 
But, by chasing down all the constants in the proof
of Theorem \ref{thm:cardcannarsa} close enough, we found that
if we ``slow down'' our data enough, i.e.~for a slow enough 
dependence of \(l\) on its last argument (the constant \(c_1\)
in~L-\ref{L-ii}(b) below), we can close this loop.
Once this loop is closed, the minimizers among the paths in \(\tilde{\Gamma}\)
are minimizers among all admissible paths, so that, in the
argument concerning sequences explained above, the approximating
sequences need not have second-order regularity and can be constructed
as in \cite{cannarsacapuani1}.

\section{Preliminaries}\label{section:preliminaries}
In this section we will fix our notation and the main hypotheses of 
our result and import a regularity result in calculus of variations 
\cite{cannarsaconstcalc}.
\subsection{Notation and hypotheses}\label{subsection:notationandhypotheses}
Whenever $\mathcal{S}$ is a complete, separable space, with metric $d_{\mathcal{S}}$,  
if $\mathscr{B}(\mathcal{S})$ denotes the set of Borel probability measures 
on $\mathcal{S}$ and $s_0\in \mathcal{S}$ is an arbitrary element, the subset 
\begin{align*}
\mathscr{P}(\mathcal{S}) = \{ \eta \in \mathscr{B}(\mathcal{S}) \ \big| \ 
\int_{\mathcal{S}} d_{\mathcal{S}}(s,s_0) \eta(ds) < \infty \}
\end{align*}
is endowed with the 1-Wasserstein distance
\begin{align*}
\mathbf{d}_{\mathcal{S}}(\eta^1,\eta^2) = \inf\limits_{\boldsymbol\Gamma(\eta^1,\eta^2)} 
\int_{\mathcal{S}\times\mathcal{S}} d_{\mathcal{S}}(s^1,s^2) \boldsymbol\gamma(ds^1,ds^2),
\end{align*}
where 
\begin{align*}
\boldsymbol\Gamma(\eta^1,\eta^2) = 
\{ \boldsymbol\gamma\in \mathscr{B}(\mathcal{S}\times\mathcal{S}) \ \big| \ \pi^1_{\#}\boldsymbol\gamma = \eta^1, 
\ \pi^2_{\#}\boldsymbol\gamma = \eta^2  \}.
\end{align*}
It is known that, if $\mathcal{S}$ happens to be compact, 
then $\mathscr{P}(\mathcal{S})$ is a compact metric space with the distance 
$\mathbf{d}_{\mathcal{S}}$; see, for instance, Remark 5.1.5 and 
Proposition 7.1.5 in \cite{gradientflows}.

If $\eta$ is a measure, $\spt(\eta)$ will denote the support of 
the measure; every Borel measure on a separable metric space has support \cite{bogachev}.

For the full extent of this paper we fix an arbitrary
positive integer $d$ and a $\s{C}^3$ bounded domain 
in $\R^d$,  i.e.~a bounded, connected open set $\Omega$ in $\R^d$ whose 
boundary $\partial \Omega$ is a surface of class $\s{C}^3$,  and $X\subset \R^d$ 
will be a fixed open bounded set including $\bar{\Omega}$. Let 
\begin{align*}
    d_{\Omega}(x) = \inf\{|x-y| \ \big| \ y\in \Omega\}, \quad d_{\Omega^c}(x) = 
	\inf\{|x-y| \ \big| \ y\notin \Omega\}
\end{align*}
and 
\begin{align*}
    b_{\Omega}(x) = d_{\Omega}(x) - d_{\Omega^c}(x),
\end{align*}
$x\in X;$ $b_{\Omega}$ is called the oriented boundary distance. 
It can be shown that there exists a positive number $\rho_0$,  such that, 
if we denote the neighbourhood of $\partial\Omega$ of radius $\rho_0$ by 
$\Sigma_{\rho_0}$, then $b_{\Omega}$ is of class $\mathcal{C}^2_b$ on $\Sigma_{\rho_0}$,  
that is, 
twice continuously differentiable with all derivatives up to second order bounded.

We fix a number $T>0$,  and a \emph{continuous} function
\begin{align*}
l: [0,T]\times X\times\R^d \times \mathscr{P}(M) & \ \longrightarrow \R
\\
(t,x,v,\nu) & \ \longmapsto l(t,x,v,\nu),
\end{align*}
where 
$$
M := \bar{\Omega}\times\R^d,
$$  
with the following properties:
\begin{enumerate}[(L-i)]
\item \label{L-i} (\textit{bound on $l$ and $Dl$ at $v=0$}) the function $l$
is $C^1$ in $(x,v)$, its derivatives $D_xl$ and $D_vl$ are continuous on
$[0,T]\times X\times\R^d\times\mathscr{P}(M)$, and there is a constant $n_1>0$ such that 
\begin{align*}
|l(t,x,0,\nu)| + |D_x l(t,x,0,\nu)| + |D_v l(t,x,0,\nu)| \leq n_1
\end{align*}
for all $(t,x,\nu)\in [0,T]\times X\times\wassspaceonM$.
\item \label{L-ii} (a) (\textit{bounds on the gradient of $D_v l$}) 
the function $D_v l$ is $C^1$ in $(t,x,v)$, and there is a constant $c > 1$ such that
\begin{align*}
\frac{1}{c}I_d \leq D^2_{vv} l (t,x,v,\nu) \leq c I_d,
\end{align*} 
where $I_d$ is the $d\times d$ identity matrix, and a constant $k_1$ such that
 \begin{align*}
| D^2_{xv}  l(t,x,v,\nu) | \leq k_1 (1+|v|),
\end{align*}
for all $(t,x,v,\nu)\in [0,T]\times X \times \R^d \times \wassspaceonM$;

(b) (\textit{Lipschitz continuity of $l$ and $D_v l$ in the measure}) the
functions $l$ and $D_v l$ are Lipschitz continuous with respect to the variable $\nu\in \mathscr{P}(M)$: there is a constant $c_1\geq 0$ such that
\begin{align*}
|l(t,x,v,\nu_1) - l(t,x,v,\nu_2)| \leq & \  c_1 \mathbf{d}_M(\nu_1,\nu_2), \quad \nu_1,\nu_2 \in \mathscr{P}(M)
\\
|D_v l(t,x,v,\nu_1) - D_v l(t,x,v,\nu_2)| \leq & \ c_1 \mathbf{d}_M(\nu_1,\nu_2), \quad \nu_1,\nu_2 \in \mathscr{P}(M)
\end{align*}
for all $(t,x,v)\in [0,T]\times X \times \R^d$; 
\item \label{L-iii} (\textit{$|v|$-dependent Lipschitz continuity in time of $l$ and $D_vl$}) there is a constant $\kappa_1$ such that
\begin{align*}
|l(t,x,v,\nu) - l(s,x,v,\nu)| \leq &\  \kappa_1 (1+|v|^2)|t-s|,
\\
|D_v l(t,x,v,\nu) - D_v l(s,x,v,\nu)| \leq &\ \kappa_1 (1+|v|)|t-s|
\end{align*}
for all $t,s\in [0,T]$,  $(x,v,\nu)\in X \times \R^d \times \wassspaceonM$. 
\end{enumerate}
We also fix a function
\begin{align*}
l_T : X \times \mathscr{P}(M) \longrightarrow &\ \R
\\
(x,\nu) \longmapsto &\ l_T(x,\nu),
\end{align*}
with the following properties:
\begin{enumerate}[(T-i)]
\item \label{T-i} For every $\nu\in\mathscr{P}(M)$ the function $l_T(\cdot,\nu)$ is continuously
differentiable on $X$, and both $l_T$ and its derivative in $x$,  $Dl_T$,  are bounded on
$X \times \mathscr{P}(M)$. 
\item \label{T-ii} Both $l_T$ and $Dl_T$ are Lipschitz continuous in $\nu$ with Lipschitz
constant $c_1$ 
(with respect to the 1-Wasserstein distance $\mathbf{d}_{M}$ on the space $\mathscr{P}(M)$).
\end{enumerate}
This implies, in particular, that $l_T$ is bounded over subsets $X\times \mathcal{K}$,  
where $\mathcal{K}$ is any compact subset of $\mathscr{P}(M)$.
We write $\|l_T\|_{\infty}$ and $\|Dl_T\|_{\infty}$ for the suprema of $|l_T|$ and $|Dl_T|$
over $X\times\mathscr{P}(M)$.

The set $M=\bar{\Omega}\times\R^d$ is endowed with the Euclidean distance of $\R^d\times\R^d$.
For $r>0$, $\bar{B}_r(0)$ denotes the closed ball of radius $r$ centered at the origin of $\R^d$.
For $a,b\in\R^d$, $a\otimes b$ denotes the $d\times d$ matrix with
entries $(a\otimes b)_{ij}=a_ib_j$, $i,j=1,\ldots,d$,
so that $(a\otimes b)w=(b\cdot w)\,a$ for $w\in\R^d$. Finally,
$L^1([0,T];\R^d)$ (resp.~$L^2([0,T];\R^d)$) is the space of Lebesgue integrable
(resp.~square integrable) functions from $[0,T]$ to $\R^d$.
For $0\leq t<T$, $AC([t,T];\bar{\Omega})$ is the set of absolutely continuous functions
$x:[t,T]\to\bar{\Omega}$, $AC^2(0,T;\bar{\Omega})$ is the set of those
$x(\cdot)\in AC([0,T];\bar{\Omega})$
with $\dot{x}(\cdot)\in L^2([0,T];\R^d)$, and $C^{1,1}([0,T];\bar{\Omega})$ is the set
of continuously differentiable
$x:[0,T]\to\bar{\Omega}$ with Lipschitz derivative.
A path $\nu_{\cdot}:[0,T]\to\mathscr{P}(M)$ is called
\emph{Lipschitz} if
\begin{align*}
\textrm{Lip}(\nu_{\cdot}) := \sup_{s\neq t}\frac{\mathbf{d}_M(\nu_s,\nu_t)}{|s-t|}
< \infty.
\end{align*} 

For every $(t,x,\nu)\in[0,T]\times X\times\mathscr{P}(\bar{\Omega}\times\R^d)$, 
let $h(t,x,\cdot,\nu)$ 
be the Legendre transform of $l(t,x,\cdot,\nu)$,  i.e.
\begin{align*}
h(t,x,p,\nu) = l^{\ast}(t,x,p,\nu) := \sup_{v\in\R^d} \{ p\cdot v - l(t,x,v,\nu)\}.
\end{align*} 
\begin{lemma}
\label{lemma:sameforH}
(\textit{The analogous properties of $h$}) It follows that:
\begin{enumerate}[(H-i)]
\item (\textit{bound on $h$ and $Dh$ at $p=0$}) the function $h$ is $C^1$ in $(x,p)$
and there is a constant $n_2>0$, which depends only on $n_1$, $c$ and $k_1$,  such that
\begin{align*}
|h(t,x,0,\nu)| + |D_x h(t,x,0,\nu)| + |D_p h(t,x,0,\nu)| \leq n_2
\end{align*}
for all $(t,x,\nu)\in [0,T]\times X\times\wassspaceonM$.
\item \label{H-ii} (a) (\textit{bounds on the gradient of $D_p h$}) the function $D_p h$ is $C^1$ in $(t,x,p)$, and
\begin{align*}
\frac{1}{c}I_d \leq D^2_{pp} h (t,x,p,\nu) \leq c I_d,
\end{align*}
and there is a constant $k_2$, which depends only on $n_1$, $c$ and $k_1$, such that
 \begin{align*}
| D^2_{xp}  h(t,x,p,\nu) | \leq k_2 (1+|p|),
\end{align*}
for all $(t,x,p,\nu)\in [0,T]\times X \times \R^d \times \wassspaceonM$;

(b)  (\textit{Lipschitz continuity of $D_p h$ in the measure}) the
 function $D_p h$ is Lipschitz continuous with respect to the variable $\nu\in \mathscr{P}(M)$ and
\begin{align*}
|D_p h(t,x,p,\nu_1) - D_p h(t,x,p,\nu_2)| \leq  \ c c_1 \mathbf{d}_M(\nu_1,\nu_2),
\quad \nu_1,\nu_2 \in \mathscr{P}(M)
\end{align*}
for all $(t,x,p)\in [0,T]\times X \times \R^d$;
\item \label{H-iii} (\textit{$|p|$-dependent Lipschitz continuity in time of $h$ and $D_p h$})
there is a constant
$\kappa_2$, which depends only on $n_1$, $c$ and $\kappa_1$, such that
\begin{align*}
|h(t,x,p,\nu) - h(s,x,p,\nu)| \leq &\  \kappa_2 (1+|p|^2)|t-s|,
\\
|D_p h(t,x,p,\nu) - D_p h(s,x,p,\nu)| \leq &\ \kappa_2 (1+|p|)|t-s|
\end{align*}
for all $t,s\in [0,T]$,  $(x,p,\nu)\in X \times \R^d \times \wassspaceonM$.
\end{enumerate}
\end{lemma}
\begin{proof}
For each fixed $\nu$, properties (H-i), (H-\ref{H-ii})(a) and (H-\ref{H-iii}) are, up to the
change of
variable $p\mapsto -p$, the properties (H0), (H1) and (H2) of the Hamiltonian
in \cite[Section 3.1]{cannarsaconstcalc},
where they are stated without proof. Since we need constants that do not depend on $\nu$, and
for the
reader's convenience, we give the argument. Unless otherwise stated, $(t,x,\nu)$ is fixed and
omitted from the notation.

By (L-\ref{L-ii})(a), the function $v\mapsto p\cdot v - l(v)$ is strictly concave and tends to
$-\infty$ as $|v|\to\infty$; hence it has a unique maximizer $v^{\ast}=v^{\ast}(p)$, which is
characterized by $D_vl(v^{\ast})=p$.
By (L-\ref{L-ii})(a) and (L-\ref{L-i}),
$\frac{1}{c}|v^{\ast}|^2\leq (D_vl(v^{\ast})-D_vl(0))\cdot v^{\ast}\leq (|p|+n_1)|v^{\ast}|$, so that
\begin{equation}\label{eq:vstarbound}
|v^{\ast}(p)| \leq c(n_1+|p|).
\end{equation}

For the bounds on the derivatives, note that since $D_vl$ is
$C^1$ in $(t,x,v)$ and $D^2_{vv}l$ is
invertible, by the implicit function theorem $v^{\ast}$ is
a $C^1$ function of $(t,x,p)$, with $D_pv^{\ast} = (D^2_{vv}l)^{-1}$ and
$D_xv^{\ast} = -(D^2_{vv}l)^{-1}D^2_{xv}l$,
where the derivatives of $l$ are evaluated at $v=v^{\ast}$.
Differentiating $h = p\cdot v^{\ast} - l(v^{\ast})$ and using $D_vl(v^{\ast})=p$,
we get (see also \cite[p.~138]{dacorognadirect})
\begin{equation}\label{eq:Dph}
D_ph(t,x,p,\nu) = v^{\ast}, \qquad D_xh(t,x,p,\nu) = -D_xl(t,x,v^{\ast},\nu).
\end{equation}
Thus $h$ is $C^1$ in $(x,p)$, $D_ph$ is $C^1$ in $(t,x,p)$,
$D^2_{pp}h = (D^2_{vv}l)^{-1}$, which gives the claimed bounds for $D^2_{pp}h$, and,
by \eqref{eq:vstarbound},
\begin{align*}
|D^2_{xp}h| \leq ck_1(1+|v^{\ast}|) \leq k_2(1+|p|), \qquad k_2 := ck_1(1+cn_1+c).
\end{align*}

To settle the bounds at $p=0$, we use (L-\ref{L-i}) and (L-\ref{L-ii})(a)
together with \eqref{eq:Dph} and \eqref{eq:vstarbound}.
We have $h(0) = -\min_v l(v)$ so, by (L-\ref{L-i}) and (L-\ref{L-ii})(a),
$n_1\geq l(0)\geq \min_vl(v)\geq \min_v\{-n_1-n_1|v|+\frac{1}{2c}|v|^2\} = -n_1-\frac{c}{2}n_1^2$.
By \eqref{eq:Dph} and \eqref{eq:vstarbound}, $|D_ph(0)| = |v^{\ast}(0)|\leq cn_1$ and
$|D_xh(0)| = |D_xl(v^{\ast}(0))| \leq n_1 + k_1(1+cn_1)cn_1$, where we used (L-\ref{L-i}) and
the bound for $D^2_{xv}l$ in (L-\ref{L-ii})(a). This gives (H-i) with
$n_2 := 2n_1 + cn_1 + \frac{c}{2}n_1^2 + ck_1n_1(1+cn_1)$.

For the time regularity, let $s,t\in[0,T]$ and let $v_t$, $v_s$ be the maximizers corresponding
to $(t,x,p,\nu)$ and
$(s,x,p,\nu)$, so that $D_vl(t,x,v_t,\nu) = p = D_vl(s,x,v_s,\nu)$. By (L-\ref{L-ii})(a) and
(L-\ref{L-iii}),
\begin{align*}
\frac{1}{c}|v_t-v_s|^2 \leq & \ \big(D_vl(t,x,v_t,\nu)-D_vl(t,x,v_s,\nu)\big)\cdot(v_t-v_s)
\\
= & \ \big(D_vl(s,x,v_s,\nu)-D_vl(t,x,v_s,\nu)\big)\cdot(v_t-v_s) \leq \kappa_1(1+|v_s|)|t-s||v_t-v_s|,
\end{align*}
so that, by \eqref{eq:Dph}, $|D_ph(t,x,p,\nu)-D_ph(s,x,p,\nu)| = |v_t-v_s| \leq c\kappa_1(1+|v_s|)|t-s|$. Moreover,
\begin{align*}
h(t,x,p,\nu)-h(s,x,p,\nu) \leq \big[p\cdot v_t - l(t,x,v_t,\nu)\big] - \big[p\cdot v_t 
- l(s,x,v_t,\nu)\big] \leq \kappa_1(1+|v_t|^2)|t-s|,
\end{align*}
and the same holds with $s$ and $t$ exchanged. By \eqref{eq:vstarbound},
$1+|v_s|\leq(1+cn_1+c)(1+|p|)$ and
$1+|v_t|^2\leq(1+2c^2n_1^2+2c^2)(1+|p|^2)$. This gives (H-\ref{H-iii}) with
$\kappa_2 := \kappa_1\max\{c(1+cn_1+c),\,1+2c^2n_1^2+2c^2\}$.

Finally, let us show the regularity in measure, i.e.~(H-\ref{H-ii})(b).
Fix $(t,x,p)$,  let $v_1$ and $v_2$ be the maximizers of the functions
$v \mapsto p\cdot v - l(t,x,v,\nu_1)$ and $v\mapsto p\cdot v -l(t,x,v,\nu_2)$ respectively.
This implies that
\begin{align*}
p = D_v l(t,x,v_1,\nu_1) = D_v l(t,x,v_2,\nu_2),
\end{align*}
and, by \eqref{eq:Dph},
$$
v_1 = D_p h(t,x,p,\nu_1), \quad v_2 = D_p h(t,x,p,\nu_2).
$$
By the lower bound on the second derivative $D^2_{vv}l$ above, we have
\begin{align*}
\frac{1}{c} |v_2-v_1|^2 \leq (D_v l(t,x,v_2,\nu) - D_v l(t,x,v_1,\nu)) \cdot (v_2 - v_1),
\end{align*}
for any $\nu$. Therefore, taking $\nu=\nu_1$,
\begin{align*}
\frac{1}{c}|v_2-v_1|^2 \leq & \ (D_v l(t,x,v_2,\nu_1) - D_v l (t,x,v_2,\nu_2))\cdot (v_2-v_1)
\\
& \ + (D_v l(t,x,v_2,\nu_2) - D_v l(t,x,v_1,\nu_1)) \cdot (v_2 - v_1)
\\
\leq & \ c_1 \mathbf{d}_M(\nu_1,\nu_2)|v_2-v_1| + (p - p)\cdot (v_2 - v_1)  = c_1
\mathbf{d}_M(\nu_1,\nu_2)|v_2-v_1|,
\end{align*}
from which the desired inequality follows.
\end{proof}
\subsubsection{Notation for spaces of paths
and controls}
Whenever $K_1$,  $K_2$ are positive numbers, we denote $K=(K_1,K_2)$. Let
\begin{align*}
\Gamma_1^{K} := \{x:[0,T] \to \bar{\Omega} \ \big| \ & \ x(\cdot) \textrm{
  is continuously differentiable}, \ |\dot{x}(t)| \leq K_1, 0\leq t\leq T,
\\
& \ \dot{x}(\cdot) \textrm{ is Lipschitz}, \ |\ddot{x}(t)|\leq K_2 \textrm{ a.e.}\}.
\end{align*}
To define the set of \emph{regular} controls, for every $x_0\in \bar{\Omega}$ we let
\begin{align*}
 \Gamma_2^{K,x_0} := \{u:[0,T] \to \bar{B}_{K_1}(0) \ \big| \ & \ u(\cdot) \textrm{ is Lipschitz}, 
\\
& \ |\dot{u}(t)|\leq K_2 \textrm{ a.e.}, \ x_0 + \int_0^t u(\tau) d\tau \in \bar{\Omega} \textrm{ for } t\in[0,T]\}
\end{align*}
for $x_0\in\bar{\Omega}$, and
\begin{align*}
\Gamma_2^{K} := \bigcup\limits_{x_0\in\bar{\Omega}} \Gamma_2^{K,x_0}.
\end{align*}
 Both $\Gamma_1^K$ and $\Gamma_2^K$ are compact with respect to the uniform metrics $d_1$ and $d_2$:
\begin{align*}
d_1(x_1,x_2) := \max\limits_{0\leq t\leq T}|x_1(t) - x_2(t)|, \quad d_2(u_1,u_2) := \max\limits_{0\leq t\leq T}|u_1(t) - u_2(t)|.
\end{align*}
Indeed, let $\{x_n\}_{n=1}^{\infty}\subset\Gamma_1^K$. The functions $\dot{x}_n$ take values in $\bar{B}_{K_1}(0)$ and
have Lipschitz constant $K_2$, and $x_n(0)\in\bar{\Omega}$. By the Arzel\`a-Ascoli theorem,
along a subsequence,
$\dot{x}_n\to u$ uniformly and $x_n(0)\to x_0$, where $u$ takes values in $\bar{B}_{K_1}(0)$
and has Lipschitz constant $K_2$,
and $x_0\in\bar{\Omega}$. Then $x_n\to x:=x_0+\int_0^{\cdot}u(\tau)d\tau$ uniformly along the
subsequence; the path $x$ takes
values in $\bar{\Omega}$, because $\bar{\Omega}$ is closed, it is continuously differentiable,
and $\dot{x}=u$. Thus $x\in\Gamma_1^K$.
The argument for $\Gamma_2^K$ is similar.
Note that $x(\cdot)\in\Gamma_1^K$ if and only if $\dot{x}(\cdot)\in\Gamma_2^{K,x(0)}$.
Let
\begin{align}
\Gamma^{K} = \{ \gamma=(x,u): [0,T]\to \R^d\times\R^d \ | \ x(\cdot) \in \Gamma_1^K, \ u(\cdot) = \dot{x}(\cdot) \}. \label{eq:GammaK}
\end{align}
Thus, for every $(x,u)\in\Gamma^K$ we have $u(\cdot)\in\Gamma_2^{K,x(0)}$ and
$x(t) = x(0)+\int_0^t u(\tau)d\tau$ for $t\in[0,T]$.
By definition, the paths $\gamma\in \Gamma^K$ are bounded. On $\Gamma^K$ we put the metric
\begin{align*}
d_{\Gamma}((x_1,u_1),(x_2,u_2)) := d_1(x_1,x_2) + d_2(u_1,u_2),
\end{align*}
so that $\Gamma^K$ is compact with respect to this metric (the argument just given shows that every sequence in $\Gamma^K$ has a subsequence that converges in $d_{\Gamma}$ to an element of $\Gamma^K$). We choose not to specify $K$ 
in the notation $d_{\Gamma}$ because the formula is the same regardless of the values of $K_1$,  $K_2$.

\subsection{Regularity of constrained minimizers}

In order to clarify the discussion that follows, we introduce the following sets:
\begin{align}
A\Gamma_1 := & \  AC([0,T];\bar\Omega),  \notag \\
\notag \\
A\Gamma_2 := & \ \{u:[0,T] \to \R^d \ \big| \ \exists x(\cdot)\in A\Gamma_1 \ | \ u(t)=\dot{x}(t) \textrm{ a.e.}\} \subset L^1([0,T];\R^d), \notag \\
\notag \\
A\Gamma& \  = \{ \gamma=(x,u): [0,T]\to \bar{\Omega}\times\R^d \ | \ x(\cdot) \in A\Gamma_1, \ u(\cdot) \in A\Gamma_2 \}. 
\label{eq:AGamma}
\end{align}
We note that $(x,u)\in A\Gamma$ does not necessarily mean that $\dot{x}=u;$ $A\Gamma_2$ 
are merely functions in $L^1([0,T]:\R^d)$ that are derivatives of paths 
in $AC([0,T];\bar{\Omega})$. 
Clearly, $\Gamma_1^K \subset A\Gamma_1$ and $\Gamma_2^K \subset A\Gamma_2$ for every $K$. 
The set $\Gamma_2^K$ consists of Lipschitz paths (with Lipschitz constant $K_2$) 
that can be used as controls for the dynamics (the corresponding integrated paths 
remain in $\bar{\Omega}$). To each $\eta\in\mathscr{P}(\Gamma^K)$
and
$x(\cdot)\in A\Gamma_1$ we may associate the cost
\begin{align*}
\int_0^T l (t,x(t),\dot{x}(t),(e_t)_{\#}\eta) dt + l_T(x(T),(e_T)_{\#}\eta),
\end{align*}
where $e_t:\Gamma^K\to \R^d\times\R^d$, $0\leq t\leq T$ are the evaluation mappings: 
\begin{align*}
e_t(\gamma) = e_t((x(\cdot),u(\cdot))) := (x(t),u(t)).
\end{align*}
This cost can be considered as a functional acting on the set of paths $A\Gamma_1$ or 
its subset $\Gamma_1^K;$ alternatively, it may be considered as a functional on the set of 
controls $A\Gamma_2$ or its subset $\Gamma_2^K$. 
Given a path\footnote{The dot as subindex is meant to distinguish the path
$t\mapsto \nu_t\equiv \nu(t)$ 
from any single measure value of this path.} $\nu_{\cdot}$ in $\wassspaceonM$ which is continuous, it is known that 
hypotheses (L-\ref{L-i}) through (L-\ref{L-iii}) above are more than enough to ensure that if 
one considers the functional
\begin{equation}\label{eq:ACprob}
J[x(\cdot);\nu_{\cdot}] := \int_0^T l (t,x(t),\dot{x}(t),\nu_t) dt + l_T(x(T),\nu_T)
\end{equation}
to be defined on the set $A\Gamma_1 = AC([0,T];\bar{\Omega})$, then there exists a minimizer of 
this functional in $AC^2(0,T;\bar{\Omega})$ where the minimization is taken with respect 
to the paths $x(\cdot)$ with a common initial point $x(0)=x_0;$ see, e.g., \cite{semiconcavehjb} and \cite[Remark 3.1]{cannarsaconstcalc}; for Lipschitz paths $\nu_{\cdot}$ this also follows from the proof of Proposition \ref{prop:adjointarc} below. 

In the present manuscript we wish to have minimizers with better regularity, 
which is why we 
introduced the sets $\Gamma^K$. For our purposes we will 
heavily rely on the following result, which is due to Cannarsa, Capuani and Cardaliaguet;
we state it in our notation (see Remark \ref{remark:Ltol} below).
\begin{thm}[{\cite[Theorem 3.1]{cannarsaconstcalc}}]\label{thm:cardcannarsa}
Let $l$, $l_T$ be subject to the conditions above, $h=l^*$,  let $\nu_{\cdot}$ be a \emph{Lipschitz} 
path in $\wassspaceonM$,  and consider the functional $J[ \ \cdot \ ;\nu_{\cdot}]$ defined above, 
on the set of $AC([0,T];\bar{\Omega})$ paths $x(\cdot)$ with a common initial point 
$x(0)=x_0$,  $x_0\in\bar{\Omega}$. Let $x(\cdot)$ be a minimizer. Then:
\begin{enumerate}[(i)]
\item $x(\cdot) \in C^{1,1}([0,T];\bar{\Omega})$, i.e.~it is continuously 
differentiable with Lipschitz derivative;
\item there exist a Lipschitz function $p:[0,T]\to \R^d$, a bounded measurable function
$\Lambda:[0,T]\to[0,\infty)$ and a number $\beta\geq 0$ such that
\begin{align*}
\dot{x}(t) = &\ D_p h(t,x(t),p(t),\nu_t),  \quad t \in [0,T],
\\
\dot{p}(t) = &\ -D_x h(t,x(t),p(t),\nu_t) +
\Lambda(t)\mathbf{1}_{\partial\Omega}(x(t))Db_{\Omega}(x(t)),
\quad \textrm{a.e. } t\in [0,T],
\\
p(T) = &\ -Dl_T(x(T),\nu_T) - \beta \mathbf{1}_{\partial\Omega}(x(T))Db_{\Omega}(x(T)),
\end{align*}
where $\mathbf{1}_{\partial\Omega}(\cdot)$ is the characteristic function of the boundary $\partial\Omega$.
\end{enumerate}
\end{thm}
In \cite{cannarsaconstcalc} it is shown, in addition, that $\beta$ is bounded by a constant and that
$\Lambda(t)$ is the value at $(t,x(t),p(t))$ of a function that does not depend on the minimizer;
we will not use the latter fact.
What we do need is to know how the function $p(\cdot)$ and its Lipschitz constant depend on the
constant $c_1$ of (L-\ref{L-ii})(b) and on the path $\nu_{\cdot}$, because they determine the
Lipschitz constants
of $x(\cdot)$ and $\dot{x}(\cdot)$. This information is not contained in the statement
of \cite[Theorem 3.1]{cannarsaconstcalc},
but it can be extracted from its proof. We do this in Proposition \ref{prop:adjointarc} below,
whose statement contains
Theorem \ref{thm:cardcannarsa}, and its proof goes through all the steps of the proof in
\cite{cannarsaconstcalc},
recording the constants at each step.

\begin{remark}\label{remark:Ltol} The result \cite[Theorem 3.1]{cannarsaconstcalc} deals with a
  time-dependent Lagrangian $f(t,x,v)$ 
	and no dependence on the measure. To adapt it to our needs, we consider the path $t\mapsto \nu_t$
  as given and put
$$
f(t,x,v) := l(t,x,v,\nu_t), \qquad g(x):=l_T(x,\nu_T), \qquad 0\leq t\leq T.
$$
The hypotheses of \cite[Section 3.1]{cannarsaconstcalc} hold for $f$ and $g$:
(f0) and (f1) there follow from the continuity of $l$ and of $\nu_{\cdot}$, (L-\ref{L-i})
and (L-\ref{L-ii})(a);
(g1) follows from (T-\ref{T-i}); and (f2), which
only requires Lipschitz continuity in time, follows from (L-\ref{L-ii})(b), (L-\ref{L-iii})
and
the Lipschitz continuity of $\nu_{\cdot}$; see \eqref{eq:timeregf} and \eqref{eq:timeregDf} below.
Also, the Hamiltonian in \cite{cannarsaconstcalc} is $H(t,x,p) = \sup_{v}\{-p\cdot v - f(t,x,v)\} = h(t,x,-p,\nu_t)$,
so that the function $p(\cdot)$ of Theorem \ref{thm:cardcannarsa} is the opposite of the
corresponding one in
\cite{cannarsaconstcalc}.
\end{remark}
\begin{center}
\textit{In what follows, $m_0$ denotes a fixed element of $\mathscr{P}(\bar{\Omega})$. }
\end{center}
\section{Estimates and a small data condition}\label{section:smalldata}
The goal of this section is Corollary \ref{cor:ifmonesmallenough}, which determines the
class $\Gamma^K$ and the
smallness condition on the constant $c_1$ of (L-\ref{L-ii})(b) under which the minimizers
of \eqref{eq:ACprob}
remain in $\Gamma^K$. It is obtained in three steps.
In Section \ref{subsection:constants} we list the constants that participate in
the estimates.
In Section \ref{subsection:adjointarc} we follow the proof
of \cite[Theorem 3.1]{cannarsaconstcalc} step by step,
recording how each bound depends on $c_1$ and on the path $\nu_{\cdot}$
(Proposition \ref{prop:adjointarc}).
In Section \ref{subsection:smalldata} we deduce Lipschitz bounds for the minimizers
and their velocities
(Lemma \ref{lemma:Ksub1}), and from them Corollary \ref{cor:ifmonesmallenough}.

\subsection{Structural constants and the frozen Lagrangian}\label{subsection:constants}
Let us say that a constant is \emph{structural} if it depends
only on $T$, $\Omega$
(through $\rho_0$ and $\|D^2b_{\Omega}\|_{\infty}:=\sup_{\Sigma_{\rho_0}}|D^2 b_{\Omega}|$),
the distance from $\bar{\Omega}$ to $\R^d\setminus X$, $c$, $n_1$, $k_1$, $\kappa_1$,
$\|l_T\|_{\infty}$ and $\|Dl_T\|_{\infty}$; thus, a structural constant depends neither on $c_1$,
nor on any path of measures, nor on the initial point of a trajectory.

nor on any path of measures, nor on the initial point of a trajectory.

\underline{\smash{Extension of the data}}.
In the proof of Proposition \ref{prop:adjointarc} we consider paths that may
leave $\bar{\Omega}$,
so we need $l$ and $l_T$ to be defined for all $x\in\R^d$.
Following \cite[Lemma 3.8]{cannarsaconstcalc}, we extend them
once and for all, by a formula that does not involve the measure variable.
Let $\xi:\R\to[0,1]$ be a fixed $C^{\infty}$ function
with $\xi=0$ on $(-\infty,\frac13]$ and $\xi=1$ on $[\frac23,\infty)$,
let
$\sigma:=\min\{\rho_0,\textrm{dist}(\bar{\Omega},\R^d\setminus X)\}$
(with $\sigma:=\rho_0$ if $X=\R^d$), and let
\begin{align*}
\chi(x):=\xi\big(b_{\Omega}(x)/\sigma\big),
\qquad \ell:=\|\xi'\|_{\infty}/\sigma.
\end{align*}
Since $\xi'$ vanishes outside $(\frac13,\frac23)$ and $b_{\Omega}$
is of class $\mathcal{C}^2$ on $\Sigma_{\rho_0}$, the function $\chi$ is of
class
$\mathcal{C}^2$ on $\R^d$ and $|D\chi|\leq\ell$. Moreover, $\chi=0$
on $\{b_{\Omega}\leq\sigma/3\}\supset\bar{\Omega}$, and $\chi=1$ on the open set
$\{b_{\Omega}>2\sigma/3\}$, which contains $\R^d\setminus X$, because
$b_{\Omega}(x)\geq\textrm{dist}(\bar{\Omega},\R^d\setminus X)\geq\sigma$
for $x\notin X$. Set
\begin{align*}
\tilde{l}(t,x,v,\nu) := \chi(x)\frac{|v|^2}{2} + \big(1-\chi(x)\big)l(t,x,v,\nu),
\qquad
\tilde{l}_T(x,\nu) := \big(1-\chi(x)\big)l_T(x,\nu),
\end{align*}
where the terms containing $1-\chi(x)$ are understood to be zero when
$\chi(x)=1$; this is the case whenever $x\notin X$.
Then $\tilde{l}=l$ and $\tilde{l}_T=l_T$ for $x\in\bar{\Omega}$. Moreover,
$\tilde{l}$ and $\tilde{l}_T$ satisfy (L-\ref{L-i}), (L-\ref{L-ii}),
(L-\ref{L-iii}) and (T-\ref{T-i}) with $X$ replaced by $\R^d$, with the same
constants $c$, $c_1$, $\kappa_1$ and $\|l_T\|_{\infty}$, and with
$n_1$, $k_1$ and $\|Dl_T\|_{\infty}$ replaced by
\begin{align*}
\tilde{n}_1:=(1+\ell)n_1, \qquad \tilde{k}_1:=k_1+\ell(n_1+1+c),
\qquad \|Dl_T\|_{\infty}+\ell\|l_T\|_{\infty}.
\end{align*}
Indeed, omitting the variables $(t,\nu)$, we find:
\begin{align*}
D_v\tilde{l} = \chi v + (1-\chi)D_vl, \qquad D^2_{vv}\tilde{l}
= \chi I_d + (1-\chi)D^2_{vv}l, \qquad
D_x\tilde{l} = D\chi\big(\frac{|v|^2}{2}-l\big) + (1-\chi)D_xl,
\end{align*}
and $D_xD_v\tilde{l} = (v-D_vl)\otimes D\chi + (1-\chi)D_xD_vl$.
Since $c>1$, the bounds $\frac1cI_d\leq D^2_{vv}\tilde{l}\leq cI_d$ obtain.
Since $|D_vl(t,x,v,\nu)|\leq n_1+c|v|$ by (L-\ref{L-i}) and (L-\ref{L-ii})(a),
we get $|D_xD_v\tilde{l}|\leq\tilde{k}_1(1+|v|)$. At $v=0$,
$|\tilde{l}|+|D_x\tilde{l}|+|D_v\tilde{l}|\leq|l|+\ell|l
+|D_xl|+|D_vl|\leq\tilde{n}_1$. Differences in $t$ and in $\nu$ of $\tilde{l}$
and
$D_v\tilde{l}$ are those of $l$ and $D_vl$ multiplied by $1-\chi\in[0,1]$.
Finally, $|\tilde{l}_T|\leq\|l_T\|_{\infty}$ and
$|D\tilde{l}_T|\leq\|Dl_T\|_{\infty}+\ell\|l_T\|_{\infty}$.
(The Lipschitz constant of $D\tilde{l}_T$ in $\nu$ may be larger than
$c_1$, but (T-\ref{T-ii}) is not used in this section.)

The functional $J[\ \cdot\ ;\nu_{\cdot}]$ on paths with values in $\bar{\Omega}$ only
involves the values of $l$ and $l_T$ at points of $\bar{\Omega}$.
Therefore replacing $l$, $l_T$ by $\tilde{l}$, $\tilde{l}_T$ changes neither this functional
nor its minimizers. \emph{Throughout this section
we therefore assume that $X=\R^d$}, and $n_1$, $k_1$ and $\|Dl_T\|_{\infty}$ denote the
constants $\tilde{n}_1$, $\tilde{k}_1$ and
$\|Dl_T\|_{\infty}+\ell\|l_T\|_{\infty}$ of the extended data.
The main point is that these constants are structural,
and so are the constants $n_2$, $k_2$
and $\kappa_2$ given by Lemma \ref{lemma:sameforH} for the extended data.
Since the two sets of
data coincide on $\bar{\Omega}$, the bounds
obtained in this section hold for the original data at points of $\bar{\Omega}$.
The constants $n_2$, $k_2$ and $\kappa_2$
of Lemma \ref{lemma:sameforH} are structural.
As in \cite[Lemma 3.1 and Remark 3.2]{cannarsaconstcalc}, we have:
\begin{lemma}\label{lemma:DvlandDxl}
For all $(t,x,\nu)\in[0,T]\times X\times\mathscr{P}(M)$ and all $v,p\in\R^d$,
\begin{align*}
|D_v l(t,x,v,\nu)| \leq & \ C(c,n_1)(1+|v|), & |D_x l(t,x,v,\nu)| \leq & \ C(k_1,n_1)(1+|v|^2),
\\
|D_p h(t,x,p,\nu)| \leq & \ C(c,n_2)(1+|p|), & |D_x h(t,x,p,\nu)| \leq & \ C(k_2,n_2)(1+|p|^2),
\end{align*}
and
\begin{equation}\label{eq:quadraticbounds}
\frac{1}{4c}|v|^2 - C_0 \leq l(t,x,v,\nu) \leq C_0(1+|v|^2), \qquad
\frac{1}{4c}|p|^2 - C_0' \leq h(t,x,p,\nu) \leq C_0'(1+|p|^2),
\end{equation}
where, for $i=1,2$,
\begin{align*}
C(c,n_i) := \max\{c,n_i\}, \quad C(k_i,n_i) := n_i + 2k_i, \quad C_0 := 2n_1 + c + c n_1^2,
\quad C_0' := 2n_2 + c + c n_2^2.
\end{align*}
All these constants are structural.
\end{lemma}
\begin{proof}
We omit the variables $(t,x,\nu)$. By (L-\ref{L-i}) and (L-\ref{L-ii})(a),
\begin{align*}
|D_v l(v)| \leq & \ |D_v l(0)| + \int_0^1 |D^2_{vv}l(\tau v)||v|d\tau \leq n_1 + c|v|,
\\
|D_x l(v)| \leq & \ |D_x l(0)| + \int_0^1 |D^2_{xv}l(\tau v)||v|d\tau \leq n_1 + k_1(1+|v|)|v| \leq n_1 + 2k_1(1+|v|^2),
\end{align*}
and, since $l(v) = l(0) + D_vl(0)\cdot v + \int_0^1 (1-\tau)D^2_{vv}l(\tau v)v\cdot v\,d\tau$,
\begin{align*}
-n_1 - n_1|v| + \frac{1}{2c}|v|^2 \leq l(v) \leq n_1 + n_1|v| + \frac{c}{2}|v|^2 .
\end{align*}
The bounds for $l$ in \eqref{eq:quadraticbounds} follow from $n_1|v|\leq \frac{1}{4c}|v|^2 + cn_1^2$
and
$n_1|v|\leq \frac{n_1}{2}(1+|v|^2)$. The proof of the bounds for $h$ is the same, with
Lemma \ref{lemma:sameforH} for the hypothesis on $h$.
\end{proof}
\begin{center}
\textit{In Sections \ref{subsection:constants} and \ref{subsection:adjointarc} we fix a Lipschitz path $\nu_{\cdot}$ in $\mathscr{P}(M)$ and a point $x_0\in\bar{\Omega}$.}
\end{center}
The constant $c_1$ and the path $\nu_{\cdot}$ will play role in the estimates
only through the product
\begin{equation*}
y := c_1\textrm{Lip}(\nu_{\cdot}).
\end{equation*}
As in Remark \ref{remark:Ltol}, we set
\begin{align*}
f(t,x,v) := l(t,x,v,\nu_t), \quad g(x) := l_T(x,\nu_T), \quad
f^{\ast}(t,x,p) := \sup_{v\in\R^d}\{p\cdot v - f(t,x,v)\} = h(t,x,p,\nu_t).
\end{align*}
It is important to identify precisely the time regularity of $f$ and of $D_pf^{\ast}$, because this is how $y$ enters.
By (L-\ref{L-iii}) and (L-\ref{L-ii})(b),
\begin{align}
    |f(t,x,v) - f(s,x,v)| \leq &\ \kappa_1 (1+|v|^2)|t-s| + c_1 \mathbf{d}_M(\nu_t,\nu_s)
    \leq \tilde{\kappa}_1(y) (1+|v|^2)|t-s|, \label{eq:timeregf}
    \\
    |D_v f(t,x,v) - D_v f(s,x,v)| \leq &\ \tilde{\kappa}_1(y) (1+|v|) |t-s|, \label{eq:timeregDf}
\end{align}
and, by (H-\ref{H-iii}) and (H-\ref{H-ii})(b),
\begin{align}
    |D_p f^{\ast}(t,x,p) - D_p f^{\ast}(s,x,p)| \leq \kappa_2(1+|p|)|t-s| + c c_1\mathbf{d}_M(\nu_t,\nu_s)
    \leq \tilde{\kappa}_2(y) (1+|p|) |t-s|,
	\label{eq:timeregDf2}
\end{align}
for all $s,t\in[0,T]$, $x\in X$ and $v,p\in\R^d$, where
\begin{equation}\label{eq:Bsub1}
\tilde{\kappa}_1(y) := \kappa_1 + y, \qquad \tilde{\kappa}_2(y) := \kappa_2 + cy.
\end{equation}
We will use the following structural constants:
\begin{equation}\label{eq:theta}
\begin{aligned}
\theta := & \ T(n_1 + C_0) + 2\|l_T\|_{\infty}, &
\bar{\delta} := & \ \min\Big\{\frac{1}{2c\,C(c,n_2)(1+\|Dl_T\|_{\infty})},1\Big\},
\\
\bar{C}_1 := & \ 2C_0'(2+\|Dl_T\|_{\infty}^2), &
\bar{C}_2 := & \ 2\|D^2b_{\Omega}\|_{\infty}C(c,n_2)^2 + 2k_2C(c,n_2) + c\,C(k_2,n_2),
\\
\bar{\rho}_0 := & \ \min\big\{\rho_0, (8c\theta T)^{1/2}\big\}, & &
\end{aligned}
\end{equation}
and the following functions of $y\geq 0$:
\begin{equation}\label{eq:functionsofy}
\begin{aligned}
C_1(y) := & \ \bar{C}_1 + \tilde{\kappa}_1(y)(T+4c\theta), &
P(y) := & \ \frac{2\sqrt{cC_1(y)}}{\bar{\delta}},
\\
a(y) := & \ \bar{C}_2 + 2\tilde{\kappa}_2(y), &
\rho(y) := & \ \min\Big\{\bar{\rho}_0, \frac{1}{8c^2a(y)}\Big\},
\\
\epsilon_1(y) := & \ \min\Big\{ \frac{\rho(y)^3}{64c^2\theta^2}, \frac{1}{4c\,a(y)(1+P(y)^2)}\Big\}, &
\bar{\Lambda}(y) := & \ \frac{1}{\epsilon_1(y)}.
\end{aligned}
\end{equation}
Explicitly,
\begin{equation}\label{eq:Lambdabar}
\bar{\Lambda}(y) = \max\Big\{ \frac{64c^2\theta^2}{\bar{\rho}_0^3}, \ 2^{15}c^8\theta^2 a(y)^3, \ 4c\,a(y)\big(1+P(y)^2\big)\Big\}.
\end{equation}
The functions $\tilde{\kappa}_1$, $\tilde{\kappa}_2$, $C_1$, $P$, $a$ and $\bar{\Lambda}$ are
continuous and nondecreasing
on $[0,\infty)$, they are defined in terms of structural constants only, and $\bar{\Lambda}(y)$
grows like $y^3$ as $y\to\infty$.
\subsection{Bounds with explicit dependence on \texorpdfstring{$c_1$}{c1} and \texorpdfstring{$\nu_{\cdot}$}{nu}}\label{subsection:adjointarc}
The following proposition is a version of Theorem \ref{thm:cardcannarsa} in which the bounds are
explicit.
Its proof follows the proof of \cite[Theorem 3.1]{cannarsaconstcalc}, which is based on a
penalization of the
constraint; what is new here is the bookkeeping of the constants, and a minor
modification in Step 4, which is needed because
in our case $D_pf^{\ast}$ is merely Lipschitz continuous in time.
\begin{prop}\label{prop:adjointarc}
Let $\nu_{\cdot}$ be a Lipschitz path in $\mathscr{P}(M)$, $y=c_1\textrm{Lip}(\nu_{\cdot})$,
$x_0\in\bar{\Omega}$,
and let $x(\cdot)$ be a minimizer of $J[\ \cdot \ ;\nu_{\cdot}]$ on the set of paths
in $AC([0,T];\bar{\Omega})$
with initial point $x_0$. Then $x(\cdot)\in C^{1,1}([0,T];\bar{\Omega})$, and there exist
a Lipschitz function $p:[0,T]\to\R^d$,
a measurable function $\lambda:[0,T]\to[0,1]$ and a number $\beta\in[0,1]$ such that
\begin{equation}\label{eq:adjointsystem}
\left.
\begin{aligned}
\dot{x}(t) = &\ D_p h(t,x(t),p(t),\nu_t),  \quad t \in [0,T],
\\
\dot{p}(t) = &\ -D_x h(t,x(t),p(t),\nu_t) +
\bar{\Lambda}(y)\lambda(t)\mathbf{1}_{\partial\Omega}(x(t))Db_{\Omega}(x(t)),
\quad \textrm{a.e. } t\in [0,T],
\\
p(T) = &\ -Dl_T(x(T),\nu_T) - \frac{\beta}{\bar{\delta}} \mathbf{1}_{\partial\Omega}(x(T))Db_{\Omega}(x(T)),
\end{aligned}
\right\}
\end{equation}
and
\begin{equation}\label{eq:C1aa}
|p(t)| \leq P(y), \quad t\in[0,T], \qquad\qquad
|\dot{p}(t)| \leq C(k_2,n_2)\big(1+P(y)^2\big) + \bar{\Lambda}(y), \quad \textrm{a.e. } t\in[0,T].
\end{equation}
\end{prop}
\begin{proof}
\textit{Step 0 (reduction to the case $X=\R^d$).} By \cite[Lemma 3.8]{cannarsaconstcalc}, applied for each $\nu\in\mathscr{P}(M)$,
the functions $l$ and $l_T$ can be modified outside a neighbourhood of
$\bar{\Omega}$ and extended to $[0,T]\times\R^d\times\R^d\times\mathscr{P}(M)$ and $\R^d\times\mathscr{P}(M)$, respectively,
in such a way that the hypotheses of Section \ref{subsection:notationandhypotheses} that are used in this section,
i.e.~(L-\ref{L-i}), (L-\ref{L-ii}), (L-\ref{L-iii}) and (T-\ref{T-i}), hold with $X=\R^d$,
with the same constants $c$, $\kappa_1$, $c_1$
and $\|l_T\|_{\infty}$, and with $n_1$, $k_1$ and $\|Dl_T\|_{\infty}$ replaced by structural
constants. Since
the functional \eqref{eq:ACprob} only involves the values of $l$ and $l_T$ at points $x\in\bar{\Omega}$, the minimizers
of \eqref{eq:ACprob} are not affected. Thus, we assume that $X=\R^d$; if this is not the case
from the outset,
then $n_1$, $k_1$ and $\|Dl_T\|_{\infty}$ in Section \ref{subsection:constants} are understood
to be the constants of the extended data.

\textit{Step 1 (the penalized problem).} For $\epsilon>0$ and $\delta>0$,  the \textit{penalized}
problem
(compare with \eqref{eq:ACprob}) is to minimize
\begin{align*}
J^{\epsilon,\delta}[x(\cdot)] := \int_0^T \big[f(t,x(t),\dot{x}(t)) + \frac{d_{\Omega}(x(t))}{\epsilon}\big] dt
+ \frac{d_{\Omega}(x(T))}{\delta} + g(x(T))
\end{align*}
over the paths $x(\cdot)\in AC([0,T];\R^d)$ with $x(0)=x_0$. There exists at least one minimizer
\cite[Section 11.2]{cesari-opt}; the set of minimizers is
denoted by $\mathcal{X}_{\epsilon,\delta}[x_0]$. Let $x(\cdot)\in\mathcal{X}_{\epsilon,\delta}[x_0]$.
Comparing
$J^{\epsilon,\delta}[x(\cdot)]$ with the cost of the constant path at $x_0$, which is at most
$n_1T + \|l_T\|_{\infty}$ by (L-\ref{L-i}),
and using the lower bound for $l$ in \eqref{eq:quadraticbounds}, we get
\begin{equation}\label{eq:energypenalized}
    \int_0^T \Big[ \frac{1}{4c}|\dot{x}(t)|^2 + \frac{1}{\epsilon} d_{\Omega}(x(t))\Big] dt \leq \theta.
\end{equation}
In particular, $|x(t)-x(s)|\leq H_0|t-s|^{1/2}$ for all $s,t\in[0,T]$, where $H_0 := \sqrt{4c\theta}$,
and,
since $d_{\Omega}$ is $1$-Lipschitz, the function $d_{\Omega}(x(\cdot))$ is $1/2$-H\"older
continuous with constant $H_0$.

We claim: \textit{if $0<\rho\leq\bar{\rho}_0$ and $\epsilon\leq\epsilon(\rho):=\frac{\rho^3}{64c^2\theta^2}$, then
$d_{\Omega}(x(t))<\rho$ for all $t\in[0,T]$, for any $\delta>0$.}
Indeed, let $t\in[0,T]$, and let $J\subset[0,T]$ be an interval of length $|J| = \rho^2/(2H_0^2)$ containing $t$; such
an interval exists because $\rho^2\leq\bar{\rho}_0^2\leq 8c\theta T = 2H_0^2T$. For $s\in J$
we have $d_{\Omega}(x(s))\geq d_{\Omega}(x(t)) - H_0|t-s|^{1/2}$.
Integrating over $J$, and using $\int_J|t-s|^{1/2}ds\leq\frac{2}{3}|J|^{3/2}$ and \eqref{eq:energypenalized},
\begin{align*}
d_{\Omega}(x(t)) \leq \frac{2}{3}H_0|J|^{1/2} + \frac{1}{|J|}\int_J d_{\Omega}(x(s))ds
\leq \frac{2}{3}H_0|J|^{1/2} + \frac{\theta\epsilon}{|J|}
= \frac{\sqrt{2}}{3}\rho + \frac{8c\theta^2\epsilon}{\rho^2} \leq \Big(\frac{\sqrt{2}}{3}+\frac{1}{8c}\Big)\rho < \rho.
\end{align*}

\textit{Step 2 (the adjoint path for the penalized problem).} Let $0<\rho\leq\bar{\rho}_0$,  $0<\epsilon\leq\epsilon(\rho)$, $0<\delta\leq 1$
and $x(\cdot)\in\mathcal{X}_{\epsilon,\delta}[x_0]$, so that $d_{\Omega}(x(t))<\rho\leq\rho_0$ for all $t$, by Step 1.
The functions $f$ and $g$ satisfy the hypotheses of \cite[Section 3.1]{cannarsaconstcalc}
(see Remark \ref{remark:Ltol}).
Therefore we can apply \cite[Lemma 3.4]{cannarsaconstcalc}, whose proof is based on the maximum
principle \cite[Theorem 8.7.1]{vinter}
and uses the restriction on $\epsilon$ only to ensure that $d_{\Omega}(x(t))\leq\rho$ for all $t$.
Written in terms of $f^{\ast}$
and of the opposite of the adjoint function of \cite{cannarsaconstcalc} (see Remark \ref{remark:Ltol}),
it says that $x(\cdot)\in C^{1,1}([0,T];\R^d)$ and that there exist a Lipschitz
function $p:[0,T]\to\R^d$, a measurable function $\lambda:[0,T]\to[0,1]$ and a number
$\beta\in[0,1]$ such that
\begin{equation}\label{eq:lipschitzarcepsilon}
\left.
\begin{aligned}
\dot{x}(t) = & \ D_p f^{\ast}(t,x(t),p(t)), \quad t\in[0,T],
\\
\dot{p}(t) = & \ -D_x f^{\ast}(t,x(t),p(t)) + \frac{\lambda(t)}{\epsilon} Db_{\Omega}(x(t)),
\quad \textrm{a.e. } t\in[0,T],
\\
p(T) = & \ -Dg(x(T)) - \frac{\beta}{\delta}Db_{\Omega}(x(T)),
\end{aligned}
\right\}
\end{equation}
where $\lambda(t)=0$ if $x(t)\in\Omega$,  $\lambda(t)=1$ if $x(t)\notin\bar{\Omega}$, $\beta=0$
if $x(T)\in\Omega$ and
$\beta=1$ if $x(T)\notin\bar{\Omega}$ (the terms containing $Db_{\Omega}$ are understood to be
zero when their coefficient is zero).
Moreover, the function
\begin{align*}
r(t) := f^{\ast}(t,x(t),p(t)) - \frac{1}{\epsilon}d_{\Omega}(x(t))
\end{align*}
is absolutely continuous and $|\dot{r}(t)|\leq\tilde{\kappa}_1(y)(1+|\dot{x}(t)|^2)$ for
a.e.~$t\in[0,T]$,
where $\tilde{\kappa}_1(y)$ is the constant of \eqref{eq:timeregf}.

From this we deduce a bound for $p(\cdot)$. By \eqref{eq:energypenalized}, $\int_0^T|\dot{r}(t)|dt\leq\tilde{\kappa}_1(y)(T+4c\theta)$.
By \eqref{eq:quadraticbounds},
\begin{align*}
\frac{1}{4c}|p(t)|^2 - C_0' - \frac{1}{\epsilon}d_{\Omega}(x(t)) \leq r(t) \leq r(T) 
+ \int_t^T|\dot{r}(s)|ds
\leq C_0'(1+|p(T)|^2) + \tilde{\kappa}_1(y)(T+4c\theta),
\end{align*}
and $|p(T)|^2\leq 2\|Dl_T\|_{\infty}^2 + 2/\delta^2$ by \eqref{eq:lipschitzarcepsilon},
because $|Db_{\Omega}|=1$. Since $\delta\leq 1$,
with $C_1(y)$ as in \eqref{eq:functionsofy} we obtain
\begin{equation}\label{eq:C1a}
|p(t)|^2 \leq 4c\Big[\frac{1}{\epsilon}d_{\Omega}(x(t)) + \frac{C_1(y)}{\delta^2}\Big], \quad t\in[0,T].
\end{equation}

\textit{Step 3 (the terminal point).} In the situation of Step 2, let $\delta=\bar{\delta}$, with $\bar{\delta}$ as in \eqref{eq:theta}.
We claim: \textit{if $x(T)\notin\bar{\Omega}$, then $Db_{\Omega}(x(T))\cdot\dot{x}(T)<0$.}
Indeed, let $e:=Db_{\Omega}(x(T))$
and $q:=-Dg(x(T))$. In this case $\beta=1$, so that $p(T) = q - e/\bar{\delta}$. By (H-\ref{H-ii})(a)
and $|e|=1$,
\begin{align*}
\big(D_pf^{\ast}(T,x(T),q-e/\bar{\delta}) - D_pf^{\ast}(T,x(T),q)\big)\cdot\big(-e/\bar{\delta}\big) \geq \frac{1}{c\bar{\delta}^2},
\end{align*}
and, by Lemma \ref{lemma:DvlandDxl}, $|D_pf^{\ast}(T,x(T),q)|\leq C(c,n_2)(1+\|Dl_T\|_{\infty})\leq 1/(2c\bar{\delta})$. Therefore,
by \eqref{eq:lipschitzarcepsilon},
\begin{align*}
e\cdot\dot{x}(T) = e\cdot D_pf^{\ast}(T,x(T),q-e/\bar{\delta}) \leq -\frac{1}{c\bar{\delta}} + \frac{1}{2c\bar{\delta}} < 0.
\end{align*}

\textit{Step 4 (the minimizers of the penalized problem remain in $\bar{\Omega}$).} Let
\begin{align*}
\rho:=\rho(y), \qquad \epsilon := \epsilon_1(y), \qquad \delta:=\bar{\delta},
\end{align*}
with the notation of \eqref{eq:theta} and \eqref{eq:functionsofy}; note that $\rho\leq\bar{\rho}_0$ and $\epsilon\leq\epsilon(\rho)$,
so that Steps 1, 2 and 3 apply. We claim: \textit{every $x(\cdot)\in\mathcal{X}_{\epsilon,\delta}[x_0]$ takes values in $\bar{\Omega}$.}
Suppose, on the contrary, that $x(t_0)\notin\bar{\Omega}$ for some $t_0$. Since
$x(0)=x_0\in\bar{\Omega}$,
there is an interval $I\ni t_0$, either of the form $I=(a,b)$ or of the form $I=(a,T]$, such
that
$x(t)\notin\bar{\Omega}$ for $t\in I$, $x(a)\in\partial\Omega$ and, in the first case, $x(b)\in\partial\Omega$. Let
\begin{align*}
\varphi(t) := d_{\Omega}(x(t)) = b_{\Omega}(x(t)), \quad t\in I.
\end{align*}
By Step 1, $0<\varphi(t)<\rho\leq\rho_0$ for $t\in I$, so that $x(t)$ lies in the region
where $b_{\Omega}$ is of class $\mathcal{C}^2_b$.
Since $\dot{x}(\cdot)$ is Lipschitz, $\varphi$ is continuously differentiable on $I$ and its derivative
$\dot{\varphi}(t) = Db_{\Omega}(x(t))\cdot\dot{x}(t)$ is Lipschitz; in particular, $\ddot{\varphi}(t)$ exists for a.e.~$t\in I$.

We now compute a lower bound for $\ddot{\varphi}$. In \cite[Lemma 3.7]{cannarsaconstcalc}
this is done by differentiating
$t\mapsto D_pf^{\ast}(t,x(t),p(t))$ by the chain rule, which requires $D_pf^{\ast}$ to be differentiable in $t$. In our case
$D_pf^{\ast}(t,x,p)=D_ph(t,x,p,\nu_t)$ is only Lipschitz continuous in $t$ (see \eqref{eq:timeregDf2}), so we argue as follows.
Let $t\in I$ be a point at which both $\dot{x}(\cdot)$ and $p(\cdot)$ are differentiable
and the second equation in
\eqref{eq:lipschitzarcepsilon} holds; almost every $t\in I$ has these properties. By the first equation in \eqref{eq:lipschitzarcepsilon},
for small $s\neq 0$,
\begin{align*}
\frac{\dot{x}(t+s)-\dot{x}(t)}{s} = & \ \frac{D_pf^{\ast}(t+s,x(t+s),p(t+s)) - D_pf^{\ast}(t,x(t+s),p(t+s))}{s}
\\
& \ + \frac{D_pf^{\ast}(t,x(t+s),p(t+s)) - D_pf^{\ast}(t,x(t),p(t))}{s}.
\end{align*}
As $s\to 0$, the left-hand side converges to $\ddot{x}(t)$, and the second term on the right-hand side converges to
$D^2_{xp}f^{\ast}\,\dot{x}(t) + D^2_{pp}f^{\ast}\,\dot{p}(t)$, where $D^2_{xp}f^{\ast}$ and $D^2_{pp}f^{\ast}$ denote the derivatives of
$D_pf^{\ast}$ with respect to $x$ and $p$, evaluated at $(t,x(t),p(t))$,
because $(x,p)\mapsto D_pf^{\ast}(t,x,p)$ is continuously differentiable by (H-\ref{H-ii})(a). Hence the first term on the right-hand side
also converges, to a vector that we denote by $\tau(t)$, and $|\tau(t)|\leq\tilde{\kappa}_2(y)(1+|p(t)|)$ by \eqref{eq:timeregDf2}.
Since $\lambda(t)=1$ on $I$, using the second equation in \eqref{eq:lipschitzarcepsilon} we obtain, with
$e:=Db_{\Omega}(x(t))$,
\begin{align*}
\ddot{\varphi}(t) = & \ D^2b_{\Omega}(x(t))\dot{x}(t)\cdot\dot{x}(t) + e\cdot\ddot{x}(t)
\\
= & \ D^2b_{\Omega}(x(t))\dot{x}(t)\cdot\dot{x}(t) + e\cdot\tau(t) + e\cdot D^2_{xp}f^{\ast}\,\dot{x}(t)
- e\cdot D^2_{pp}f^{\ast}\,D_xf^{\ast} + \frac{1}{\epsilon}e\cdot D^2_{pp}f^{\ast}\,e .
\end{align*}
The last term is at least $1/(c\epsilon)$, by (H-\ref{H-ii})(a) and $|e|=1$. For the
remaining terms we use
$|\dot{x}(t)| = |D_pf^{\ast}|\leq C(c,n_2)(1+|p(t)|)$ and $|D_xf^{\ast}|\leq C(k_2,n_2)(1+|p(t)|^2)$ (Lemma \ref{lemma:DvlandDxl}),
$|D^2_{xp}f^{\ast}|\leq k_2(1+|p(t)|)$ and $|D^2_{pp}f^{\ast}|\leq c$ (H-\ref{H-ii})(a), the bound for $\tau(t)$, and the
inequalities $(1+|p|)^2\leq 2(1+|p|^2)$ and $1+|p|\leq 2(1+|p|^2)$. With $\bar{C}_2$ as in \eqref{eq:theta} and $a(y)$ as in \eqref{eq:functionsofy}, we get
\begin{equation}\label{eq:obtainedineq}
\ddot{\varphi}(t) \geq \frac{1}{c\epsilon} - a(y)\big(1+|p(t)|^2\big).
\end{equation}
By \eqref{eq:C1a} and $\varphi(t)<\rho$, we have $1+|p(t)|^2\leq 1 + P(y)^2 + 4c\rho/\epsilon$.
By the choice of $\rho$ and $\epsilon$,
we have $4c\,a(y)\rho\leq 1/(2c)$ and $a(y)(1+P(y)^2)\leq 1/(4c\epsilon)$. Therefore \eqref{eq:obtainedineq} gives
\begin{align*}
\ddot{\varphi}(t) \geq \frac{1}{c\epsilon} - \frac{1}{4c\epsilon} - \frac{1}{2c\epsilon} = \frac{1}{4c\epsilon} > 0 \qquad \textrm{for a.e. } t\in I.
\end{align*}
Since $\dot{\varphi}$ is Lipschitz, this shows that $\dot{\varphi}$ is increasing on $I$, i.e.~that $\varphi$ is convex on $I$.
If $I=(a,b)$, then $\varphi$ extends continuously to $[a,b]$ with $\varphi(a)=\varphi(b)=0$, and convexity gives
$\varphi\leq 0$ on $I$, a contradiction. If $I=(a,T]$, then $\dot{\varphi}(T)<0$ by Step 3, hence
$\dot{\varphi}<0$ on $I$, so that $\varphi(T)<\varphi(t)$ for all $t\in(a,T)$; letting $t\to a$ we get
$\varphi(T)\leq 0$, again a contradiction. This proves the claim.

\textit{Step 5 (conclusion).} Let $\epsilon=\epsilon_1(y)$ and $\delta=\bar{\delta}$. Since $J^{\epsilon,\delta}$ coincides with
$J[\ \cdot \ ;\nu_{\cdot}]$ on the paths that take values in $\bar{\Omega}$, and since the minimizers of $J^{\epsilon,\delta}$
are such paths (Step 4), the two problems have the same minimal value. Hence every
minimizer $x(\cdot)$ of
$J[\ \cdot \ ;\nu_{\cdot}]$ on the paths in $AC([0,T];\bar{\Omega})$ with initial point $x_0$
belongs to $\mathcal{X}_{\epsilon,\delta}[x_0]$,
and Step 2 applies to it. Since $d_{\Omega}(x(t))=0$ for all $t$, \eqref{eq:lipschitzarcepsilon}
becomes
\eqref{eq:adjointsystem} (recall that $1/\epsilon_1(y)=\bar{\Lambda}(y)$, that $\lambda(t)=0$ if $x(t)\in\Omega$
and that $\beta=0$ if $x(T)\in\Omega$), and \eqref{eq:C1a} becomes $|p(t)|^2\leq 4cC_1(y)/\bar{\delta}^2 = P(y)^2$.
Finally, by the second equation in \eqref{eq:adjointsystem}, Lemma \ref{lemma:DvlandDxl}, $0\leq\lambda(t)\leq 1$ and $|Db_{\Omega}|=1$,
\begin{align*}
|\dot{p}(t)| \leq C(k_2,n_2)\big(1+|p(t)|^2\big) + \bar{\Lambda}(y) \leq C(k_2,n_2)\big(1+P(y)^2\big) + \bar{\Lambda}(y)
\end{align*}
for a.e.~$t\in[0,T]$.
\end{proof}
\begin{remark}
The proof shows, in addition, that a minimizer of $J[\ \cdot \ ;\nu_{\cdot}]$ on the set of paths in $AC([0,T];\bar{\Omega})$
with initial point $x_0$ exists, since the minimizers of the penalized problem with $\epsilon=\epsilon_1(y)$ and
$\delta=\bar{\delta}$ are such minimizers.
\end{remark}
\subsection{Lipschitz bounds for the minimizers and the small data condition}\label{subsection:smalldata}
For $y\geq 0$ let
\begin{align}
\Phi_1(y) := & \ C(c,n_2)\big( 1 + P(y)\big), \label{eq:Lipcnstofx}
\\
\Phi_2(y) := & \ \big(1+P(y)\big)\big[ \tilde{\kappa}_2(y) + k_2 \Phi_1(y) \big]
+ c\big[ C(k_2,n_2)\big(1+P(y)^2\big) + \bar{\Lambda}(y) \big], \label{eq:Phi2}
\end{align}
with the notation of \eqref{eq:Bsub1} and \eqref{eq:functionsofy}.
The functions $\Phi_1,\Phi_2:[0,\infty)\to[0,\infty)$ are continuous and nondecreasing,
and they are defined in terms of structural constants only; $\Phi_2(y)$ grows like $y^3$ as $y\to\infty$,
because so does $\bar{\Lambda}(y)$.
\begin{lemma}\label{lemma:Ksub1}
Let $l$, $l_T$ satisfy the hypotheses of section \ref{subsection:notationandhypotheses}, let
$\nu_{\cdot}$ be a Lipschitz path in $\mathscr{P}(M)$ and $x_0\in \bar{\Omega}$.
Then any minimizer $x(\cdot)$ of
\begin{align*}
\min\{J[x(\cdot);\nu_{\cdot}] \ \big| \ x(\cdot) \in A\Gamma_1, \ x(0)=x_0\}
\end{align*}
belongs to $C^{1,1}([0,T];\bar{\Omega})$ and satisfies
\begin{align*}
|\dot{x}(t)| \leq \Phi_1(c_1\textrm{Lip}(\nu_{\cdot})), \quad 0\leq t\leq T, \qquad \textrm{and} \qquad
\textrm{Lip}(\dot{x}(\cdot)) \leq \Phi_2(c_1\textrm{Lip}(\nu_{\cdot})).
\end{align*}
In particular, $(x(\cdot),\dot{x}(\cdot))$ lies in $\Gamma^{K}$,  where $K=(K_1,K_2)$, whenever
$K_1\geq \Phi_1(c_1\textrm{Lip}(\nu_{\cdot}))$ and $K_2\geq \Phi_2(c_1\textrm{Lip}(\nu_{\cdot}))$.
\end{lemma}
\begin{proof}
Let $y=c_1\textrm{Lip}(\nu_{\cdot})$ and let $p(\cdot)$ be as in Proposition \ref{prop:adjointarc}. By the first
equation in \eqref{eq:adjointsystem}, Lemma \ref{lemma:DvlandDxl} and \eqref{eq:C1aa},
\begin{align*}
|\dot{x}(t)| \leq C(c,n_2)\big( 1 + |p(t)|\big) \leq \Phi_1(y), \quad 0\leq t\leq T.
\end{align*}
Next, let $0\leq t<s\leq T$. With $f^{\ast}(t,x,p)=h(t,x,p,\nu_t)$ as in Section \ref{subsection:constants},
\begin{align*}
|\dot{x}(s)-\dot{x}(t)|
\leq & \  |D_p f^{\ast}(s,x(s),p(s)) - D_p f^{\ast}(t,x(s),p(s))|
\\
& + |D_p f^{\ast}(t,x(s),p(s)) - D_p f^{\ast}(t,x(t),p(s))|
\\
& + |D_p f^{\ast}(t,x(t),p(s)) - D_p f^{\ast}(t,x(t),p(t))|.
\end{align*}
The first term is at most $\tilde{\kappa}_2(y)(1+P(y))|t-s|$, by \eqref{eq:timeregDf2} and \eqref{eq:C1aa}.
The second term equals $|\int_t^s D^2_{xp}f^{\ast}(t,x(r),p(s))\dot{x}(r)dr|$, which is at most
$k_2(1+P(y))\Phi_1(y)|t-s|$, by (H-\ref{H-ii})(a) and the bound for $\dot{x}(\cdot)$ just obtained. The third term is at most
$c|p(s)-p(t)|$, by (H-\ref{H-ii})(a), and $|p(s)-p(t)|\leq [C(k_2,n_2)(1+P(y)^2)+\bar{\Lambda}(y)]|t-s|$ by \eqref{eq:C1aa}.
Adding up, $|\dot{x}(s)-\dot{x}(t)|\leq\Phi_2(y)|s-t|$. The last assertion follows from the definition of $\Gamma^K$,
since $x(\cdot)$ takes values in $\bar{\Omega}$.
\end{proof}

Suppose now that $\eta\in \mathscr{P}(\Gamma^K)$,
and that $\nu_{\cdot}$ is defined by $\nu_t = (e_t)_{\#}\eta$,  $0\leq t\leq T$.
Then $\nu_{\cdot}$ is a Lipschitz path and
\begin{equation}\label{eq:LipnuK}
\textrm{Lip}(\nu_{\cdot}) \leq K_1 + K_2.
\end{equation}
Indeed, for $s,t\in[0,T]$ the measure $(e_s,e_t)_{\#}\eta$ has marginals $\nu_s$ and $\nu_t$, and
$|e_t(\gamma)-e_s(\gamma)|\leq (K_1+K_2)|t-s|$ for every $\gamma\in\Gamma^K$, so that
$\mathbf{d}_M(\nu_s,\nu_t)\leq\int_{\Gamma^K}|e_t(\gamma)-e_s(\gamma)|\eta(d\gamma)\leq(K_1+K_2)|t-s|$.
\begin{cor}\label{cor:ifmonesmallenough}
Let $l$, $l_T$ satisfy the hypotheses of section \ref{subsection:notationandhypotheses}, and set
\begin{align*}
K_1 := \Phi_1(1), \qquad K_2 := \Phi_2(1), \qquad K:=(K_1,K_2), \qquad \bar{c} := \frac{1}{K_1+K_2},
\end{align*}
which are structural constants.
If $c_1\leq \bar{c}$, then for any $x_0\in\bar{\Omega}$ and any $\eta\in\mathscr{P}(\Gamma^K)$,
setting
$$
\nu_t := (e_t)_{\#}\eta, \quad 0\leq t\leq T,
$$
any minimizer $x(\cdot)$ of
\begin{align*}
\min\{J[x(\cdot);\nu_{\cdot}] \ \big| \ x(\cdot) \in A\Gamma_1, \ x(0)=x_0\}
\end{align*}
satisfies: $(x(\cdot),\dot{x}(\cdot))\in \Gamma^K$.
\end{cor}
\begin{proof}
By \eqref{eq:LipnuK}, $c_1\textrm{Lip}(\nu_{\cdot})\leq \bar{c}(K_1+K_2) = 1$.
Since $\Phi_1$ and $\Phi_2$ are nondecreasing, Lemma \ref{lemma:Ksub1}
says that $|\dot{x}(t)|\leq\Phi_1(1)=K_1$ for all $t$ and that $\Phi_2(1)=K_2$ is a Lipschitz constant of the
velocity $\dot{x}(\cdot)$.
\end{proof}

Note that there is no conflict between the two latter propositions in the space
where $\nu_{\cdot}$ is taken, since any probability measure $\nu$ on
$\bar{\Omega}\times \bar{B}_{K_1}(0)$ can be regarded as a probability measure $\bar\nu$
on $\bar{\Omega}\times \R^d$ by setting $\bar\nu(E) := \nu(E\cap(\bar{\Omega}\times\bar{B}_{K_1}(0))$
for every Borel subset $E$ of $\bar{\Omega}\times \R^d$.

\begin{remark}\label{remark:smallnesscondition}
The smallness condition appearing in
Corollary \ref{cor:ifmonesmallenough} ensures,
as explained in the introduction,
that there is a uniform $K$ such that minimizers
remain in $\Gamma^K$ when the input measure
(on the space of state paths and control paths)
is supported in \(\Gamma^K\). Heuristically
we could explain this as follows. Suppose we solve the minimization problem with a trial bound
\(K\) for a myriad of agents.
The population’s velocity distribution then has
Lipschitz constant of order $K$.
Our estimates bound the acceleration of the minimizers by a quantity that grows like
$(c_1K)^3$
  plus lower‑order terms. If
$c_1$
  is not small, this output bound is larger
  than the input $K$, so we must increase $K$. But then the input Lipschitz constant grows
  and the cubic term makes the required output bound grow even faster, leading to a runaway.
  The smallness of
$c_1$ takes care of this ``cubic feedback'' and allows a non‑empty window of admissible $K$.
We stress the order of the choices in Corollary \ref{cor:ifmonesmallenough}: $K$ is fixed first, in terms of
structural constants only, and then $c_1$ is required to be small in terms of $K$; the
estimates do not allow to enlarge $K$ while keeping $c_1$ fixed.
\end{remark}
Henceforth we suppose that $c_1\leq\bar{c}$ and that $K$ and $\bar{c}$ are as in Corollary \ref{cor:ifmonesmallenough}.
\section{Notion of mild solution, existence and uniqueness}\label{section:constrainedequilibrium}
Recall that $m_0\in \mathscr{P}(\bar{\Omega})$ is fixed, and that $c_1\leq\bar{c}$ and $K$ are as in
Corollary \ref{cor:ifmonesmallenough}. We denote by $\wassspaceonGammamKzero$ the subset of 
those $\eta\in\wassspaceonGammaK$ such that $(\pi^1\circ e_0)_{\#}\eta = m_0$, where $\pi^1(x,v):=x$.
Let $\eta\in\wassspaceonGammaK$ and $x_0\in\bar{\Omega}$, and consider the functional
\begin{equation}\label{eq:functionalu}
I^{x_0}[u(\cdot);\eta]:= \int_0^T l(t,x^u_{x_0}(t),u(t),(e_t)_{\#}\eta) dt + l_T(x^u_{x_0}(T),(e_T)_{\#}\eta),
\end{equation}
defined for $u(\cdot)\in A\Gamma_2^{x_0}$,
where
\begin{align*}
A\Gamma_2^{x_0} = \{ u(\cdot) \in A\Gamma_2 \ \big| \ x_0 +\int_0^t u(\tau)d\tau \in \bar{\Omega} \ \forall t\in[0,T]\}
\end{align*}
and 
\begin{equation}\label{eq:defnofxu}
x^u_{x_0}(t) := x_0 +\int_0^t u(\tau)d\tau.
\end{equation} 
Note that $I^{x_0}[u(\cdot);\eta] = J[x^u_{x_0}(\cdot);\nu_{\cdot}]$ with $\nu_t=(e_t)_{\#}\eta$, and that
$I^{x_0}[u(\cdot);\eta]\in(-\infty,+\infty]$, by \eqref{eq:quadraticbounds}. The elements of $A\Gamma_2^{x_0}$ are
functions in $L^1([0,T];\R^d)$, so two of them are identified if they coincide a.e.; with this identification,
$\Gamma_2^{K,x_0}\subset A\Gamma_2^{x_0}$.
Inspired by \cite{cannarsacapuani1}, we make the following definition.
\begin{defn}\label{defn:CEMFGC}
We say that $\eta\in \wassspaceonGammamKzero$ is a \textit{CEMFGC} 
(a \guillemotleft constrained equilibrium of mean field game of \underline{controls}\guillemotright) if:
\begin{center}
for 
$\eta$-a.e.~$(x(\cdot),u(\cdot))\in\Gamma^K$: $\quad I^{x(0)}[u(\cdot);\eta] \leq I^{x(0)}[\tilde{u}(\cdot);\eta]$ for all $\tilde{u}(\cdot) \in \Gamma_2^{K,x(0)}$.
\end{center}
\end{defn}
Similarly to \cite{cannarsacapuani1}, we set 
\begin{align*}
\Gamma^{\eta}[x_0] := \big\{ u(\cdot) \in \Gamma_2^{K,x_0} \ \big| \ I^{x_0}[u(\cdot);\eta] 
\leq I^{x_0}[\tilde{u}(\cdot);\eta] \ \forall \tilde{u}\in \Gamma_2^{K,x_0} \big\}, \quad \ x_0\in\bar{\Omega}, 
\ \eta\in\wassspaceonGammamKzero,
\end{align*}
i.e., $\Gamma^{\eta}[x_0]$ is the set of optimal controls in $\Gamma_2^{K,x_0}$
for the functional $I^{x_0}[\cdot;\eta]$. The next lemma, which is a consequence of the small data condition,
says that these are precisely the controls that are optimal among \emph{all} admissible controls. This will allow us
to compare them with competitors that are not regular.
\begin{lemma}\label{lemma:optimaloverAC}
Let $\eta\in\wassspaceonGammaK$ and $x_0\in\bar{\Omega}$. Then the functional $I^{x_0}[\cdot;\eta]$ attains its minimum
over $A\Gamma_2^{x_0}$, every minimizer belongs to $\Gamma_2^{K,x_0}$, and
\begin{equation}\label{eq:optimaloverAC}
\emptyset\neq\Gamma^{\eta}[x_0] = \big\{ u(\cdot) \in A\Gamma_2^{x_0} \ \big| \ I^{x_0}[u(\cdot);\eta]
\leq I^{x_0}[\tilde{u}(\cdot);\eta] \ \forall \tilde{u}\in A\Gamma_2^{x_0} \big\}.
\end{equation} 
\end{lemma}
\begin{proof}
Let $\nu_t:=(e_t)_{\#}\eta$. A minimizer of $I^{x_0}[\cdot;\eta]$ over $A\Gamma_2^{x_0}$ exists, as recalled after \eqref{eq:ACprob}.
Let $u^{\ast}(\cdot)$ be any such minimizer and let $m^{\ast}:=I^{x_0}[u^{\ast}(\cdot);\eta]$. The path
$x^{\ast}(\cdot):=x^{u^{\ast}}_{x_0}(\cdot)$ is a minimizer of $J[\ \cdot \ ;\nu_{\cdot}]$ on the set of paths in $AC([0,T];\bar{\Omega})$
with initial point $x_0$. By Corollary \ref{cor:ifmonesmallenough}, $(x^{\ast}(\cdot),\dot{x}^{\ast}(\cdot))\in\Gamma^K$, that is,
$u^{\ast}(\cdot)=\dot{x}^{\ast}(\cdot)\in\Gamma_2^{K,x_0}$. Since $\Gamma_2^{K,x_0}\subset A\Gamma_2^{x_0}$,
we have $I^{x_0}[u(\cdot);\eta]\geq m^{\ast}$ for every $u(\cdot)\in\Gamma_2^{K,x_0}$, with equality for $u(\cdot)=u^{\ast}(\cdot)$.
Hence the minimum of $I^{x_0}[\cdot;\eta]$ over $\Gamma_2^{K,x_0}$ equals $m^{\ast}$, and
$\Gamma^{\eta}[x_0] = \{u(\cdot)\in\Gamma_2^{K,x_0} \ | \ I^{x_0}[u(\cdot);\eta]=m^{\ast}\}$.
This set contains $u^{\ast}(\cdot)$,
its elements are minimizers of $I^{x_0}[\cdot;\eta]$ over $A\Gamma_2^{x_0}$ and, as we have just seen, every minimizer
of $I^{x_0}[\cdot;\eta]$ over $A\Gamma_2^{x_0}$ belongs to it.
\end{proof}

Recall that $e_t : \Gamma^K \to\bar{\Omega}\times\R^d$ denotes the evaluation mappings, 
i.e.~$e_t(\gamma(\cdot)) = e_t(x(\cdot),u(\cdot)) = (x(t),u(t))$. 
In particular, $\pi^1\circ e_0$ 
is the $\bar{\Omega}$-valued mapping that assigns to each pair $(x(\cdot),u(\cdot))$ in $\Gamma^K$ 
the initial point of the trajectory, $x(0)$. Given $\eta\in\wassspaceonGammamKzero$, let
\begin{align*}
\mathcal{G}(\eta) := \{ (x(\cdot),u(\cdot))\in\Gamma^K \ \big| \ u(\cdot)\in\Gamma^{\eta}[x(0)] \},
\end{align*}
the set of pairs in $\Gamma^K$ whose control is optimal.
It is nonempty by Lemma \ref{lemma:optimaloverAC}, and it
is a closed subset of $\Gamma^K$ by Lemma \ref{lemma:Gammacetaisclosed} below, because
it is the preimage of the closed set
$\{(x_0,u(\cdot))\in\bar{\Omega}\times\Gamma_2^K \ | \ u(\cdot)\in\Gamma^{\eta}[x_0]\}$ under the continuous mapping
$(x(\cdot),u(\cdot))\mapsto(x(0),u(\cdot))$.
Similarly to \cite{cannarsacapuani1}, we define
\begin{align}
E:\wassspaceonGammamKzero & \ \longrightarrow 2^{\wassspaceonGammamKzero} \notag
\\
\eta & \ \longmapsto E(\eta) := \{ \hat{\eta} \in
\wassspaceonGammamKzero \ \big| \ \hat{\eta}(\mathcal{G}(\eta)) = 1\}. \label{eq:defnofE}
\end{align}
\begin{remark}\label{remark:verifyequivalence}
Given $m_0\in\wassspaceonOmega$, $\eta\in \wassspaceonGammamKzero$,  we have $\eta\in E(\eta)$
if and only if $\eta$  is a CEMFGC. Indeed, since $u(\cdot)\in\Gamma_2^{K,x(0)}$ for every $(x(\cdot),u(\cdot))\in\Gamma^K$,
the set $\mathcal{G}(\eta)$ is precisely the set of those $(x(\cdot),u(\cdot))\in\Gamma^K$ for which the inequalities
of Definition \ref{defn:CEMFGC} hold.
\end{remark}
\subsection{Nonemptiness of the fixed-point mapping}

We must check that $E(\eta)\neq\emptyset$ for each 
$\eta\in \wassspaceonGammamKzero$.  
The sets $\Gamma^{\eta}[x]$ aren't empty but may have more than one element. 
Let $x_0\mapsto s(x_0)\in\Gamma^{\eta}[x_0]$ be, for the moment, an arbitrary 
selection 
of the set-valued mapping 
$\bar{\Omega}\ni x_0\mapsto \Gamma^{\eta}[x_0]\subset \Gamma_2^K$. 
Let $\eta'\in\wassspaceonGammamKzero$ be the probability measure defined by:
\begin{align*}
\int_{\Gamma^K} \phi(\gamma) \eta'(d\gamma)
= \int_{\bar{\Omega}}\int_{\Gamma^K} \phi(\gamma) \delta_{(x^{s(x_{{0}})}_{x_0},s(x_{{0}}))}(d\gamma)m_0(dx_0)
\end{align*}
for every Borel function $\phi:\Gamma^K\to [0,\infty]$ 
(recall the notation $x^u_{x_0}$ from (\ref{eq:defnofxu})). 
For this definition to make sense,
the family $\{ \delta_{(x^{s(x_{{0}})}_{x_0},s(x_{{0}}))} \}_{x_0\in\bar{\Omega}}$
must be a Borel family, i.e., the mapping
\begin{align*}
\bar{\Omega} &\ \longrightarrow \wassspaceonGammamKzero
\\
x_0 &\ \longmapsto \delta_{(x^{s(x_{{0}})}_{x_0},s(x_{{0}}))}
\end{align*}
should be Borel, which means, by definition, that for every open 
subset $W\subset \Gamma^K$,  the function
 $\bar{\Omega}\ni x_0\mapsto \delta_{(x^{s(x_{{0}})}_{x_0},s(x_{{0}}))}(W)$ 
 should be Borel. This function takes on values in the set $\{0,1\}$,  
 so it suffices to show that the preimage of $1$,
\begin{align*}
\{x_0\in \bar{\Omega} \ \big| \ \delta_{(x^{s(x_{{0}})}_{x_0},s(x_{{0}}))}(W) = 1\} 
= \{ x_0\in \bar{\Omega} \ \big| \ (x^{s(x_0)}_{x_0},s(x_0)) \in W\},
\end{align*}
is a Borel set whenever $W\subset \Gamma^K$ is a Borel set. 
Thus, the selection $s$ cannot be arbitrary. 
We provide next the necessary background for the reader's convenience. 
The classic reference is \cite{aubinsetvalued}.

\subsubsection*{Set-valued maps} 

\textbf{Definition.} A set-valued map $F:X\leadsto Y$ (i.e.~$F:X\to 2^Y$), 
where $X$ and $Y$ are metric spaces, is called 
\textit{upper semicontinuous} (u.s.c.) at 
$x\in\textrm{Dom}(F):=\{x\in X \ | \ F(x)\neq\emptyset\}$
 if for every open neighbourhood $O$ of $F(x)$ 
 there exists an open ball $B\subset X$ 
 containing $x$ such that $x'\in B \Rightarrow F(x') \subset O$. 

It is easily checked that if $\textrm{Dom}(F)$ 
is closed, then $F:X\leadsto Y$ is
 u.s.c.~if and only if for every closed 
 $C\subset Y$,  $F^{-1}(C)$ is closed; 
 this criterion will be used a few lines below.

It is also easy to verify the following: 

\textit{\textbf{Fact}. If $Y$ is compact, and the $\textrm{Graph}(F) := \{(x,y)\in X\times Y \ | \ y\in F(x)\}$
 is closed, then $F$ is an u.s.c.~map with closed domain and closed images. }

 \textbf{Definition.} 
 Let $(\Theta,\mathcal{A})$ 
 be a measurable space, $Y$ 
 a metric space. A set-valued map 
 $F:\Theta\leadsto Y$ is called 
 \textit{measurable} if for every 
 open set $O\subset Y$,  $F^{-1}(O)\in \mathcal{A}$.

\textit{\textbf{Fact}. Let $X$ be a metric space
 and $\mathcal{A}$ a $\sigma$-algebra on $X$ that 
 includes all the open sets. Suppose 
 that the set-valued map $F:X\leadsto Y$ is 
 such that $F^{-1}(C)$ is closed whenever 
 $C\subset Y$ is closed. Then $F$ is measurable.}

\textbf{(}
\textit{\textbf{Proof}}. 
Fix an open set $O\subset Y$. 
For every $n\in \N$,  let $C_n$ 
be the closed set 
$\{y\in Y \ | \ d(y,Y\setminus O) \geq 1/n\}$. 
Then $O = \cup_{n=1}^{\infty} C_n$
 and $F(x)\cap O \neq \emptyset$ 
 if and only if $x\in \cup_{n=1}^{\infty} F^{-1}(C_n)$,  
 and the latter set belongs to $\mathcal{A}$. 
 $\square$ 
 \textbf{)}

\textit{\textbf{Fact} (\cite[Thm 8.1.3]{aubinsetvalued}). 
Let $Y$ be a complete separable metric space, $(\Theta,\mathcal{A})$ 
a measurable space, $F$ a measurable set-valued map from $\Theta$ to 
closed nonempty subsets of $Y$. Then there exists a measurable selection of $F$.} 
 
By the foregoing, and since $\Gamma_2^K$ is a compact set, a 
measurable selection $x_0\mapsto s(x_0)\in\Gamma^{\eta}[x_0]\subset \Gamma_2^K$ 
is possible if $\textrm{Graph}(\Gamma^{\eta})$ is closed. We
show a slightly more general statement, Lemma \ref{lemma:Gammacetaisclosed} below. Its proof uses
Lemma \ref{lemma:optimaloverAC} and an approximation of constrained trajectories. The latter is obtained by
projecting onto $\bar{\Omega}$, as in \cite[Section 3.1]{cannarsacapuani1} (see Proposition 3.1 there);
since we only need it for trajectories with continuous velocity, and for the reader's convenience, we give the proof.

\subsubsection*{Projection onto \texorpdfstring{$\bar{\Omega}$}{the closure of the domain} and approximation of constrained trajectories}

Let $N_{\rho_0}:=\{z\in\R^d \ | \ d_{\Omega}(z)<\rho_0\}$ and
\begin{align*}
\Pi_{\Omega}:N_{\rho_0}\to\R^d, \qquad \Pi_{\Omega}(z) := z \ \textrm{ if } z\in\bar{\Omega}, \qquad
\Pi_{\Omega}(z) := z - b_{\Omega}(z)Db_{\Omega}(z) \ \textrm{ if } z\notin\bar{\Omega}.
\end{align*}
If $z\in N_{\rho_0}\setminus\bar{\Omega}$, then $z\in\Sigma_{\rho_0}$, $b_{\Omega}(z)=d_{\Omega}(z)$, and
\begin{equation}\label{eq:gradb}
|Db_{\Omega}(z)|=1, \qquad \Pi_{\Omega}(z)\in\partial\Omega, \qquad |\Pi_{\Omega}(z)-z| = d_{\Omega}(z).
\end{equation}
Indeed, let $q\in\partial\Omega$ be such that $|z-q|=d_{\Omega}(z)$ and let $e:=(z-q)/|z-q|$. For small $k\in\R^d$ we have
$b_{\Omega}(z+k) = d_{\Omega}(z+k)\leq|z+k-q| = d_{\Omega}(z) + e\cdot k + o(|k|)$. Since $b_{\Omega}$ is differentiable at $z$,
this implies $Db_{\Omega}(z)\cdot k\leq e\cdot k$ for all $k\in\R^d$, i.e.~$Db_{\Omega}(z)=e$, and then $\Pi_{\Omega}(z)=q$.
In particular, $\Pi_{\Omega}$ takes values in $\bar{\Omega}$ and is continuous.
\begin{lemma}\label{lemma:projection}
Let $z:[\alpha,\beta]\to N_{\rho_0}$ be Lipschitz with constant $L$, and let $w(t):=\Pi_{\Omega}(z(t))$, $\alpha\leq t\leq\beta$.
Then $w(\cdot)$ takes values in $\bar{\Omega}$, $|w(t)-z(t)|=d_{\Omega}(z(t))$ for all $t$, $w(\cdot)$ is Lipschitz with constant $C_{\Omega}L$, where
\begin{align*}
C_{\Omega} := 1+\rho_0\|D^2b_{\Omega}\|_{\infty},
\end{align*}
and, for a.e.~$t\in[\alpha,\beta]$,
\begin{equation}\label{eq:derivativeofprojection}
\dot{w}(t) =
\begin{cases}
\dot{z}(t) & \textrm{if } z(t)\in\bar{\Omega},
\\
\dot{z}(t) - \big(Db_{\Omega}(z(t))\cdot\dot{z}(t)\big)Db_{\Omega}(z(t)) - b_{\Omega}(z(t))D^2b_{\Omega}(z(t))\dot{z}(t) & \textrm{if } z(t)\notin\bar{\Omega}.
\end{cases}
\end{equation}
\end{lemma}
\begin{proof}
The first two assertions follow from \eqref{eq:gradb}. Let
$U:=\{t\in[\alpha,\beta] \ | \ z(t)\notin\bar{\Omega}\}$, which is
relatively open in $[\alpha,\beta]$, and $C:=[\alpha,\beta]\setminus U$.
On the open set $N_{\rho_0}\setminus\bar{\Omega}$
the map $\Pi_{\Omega}$ is continuously differentiable and its derivative
is $I_d - Db_{\Omega}\otimes Db_{\Omega} - b_{\Omega}D^2b_{\Omega}$.
Since $|Db_{\Omega}|=1$, the matrix $I_d - Db_{\Omega}\otimes Db_{\Omega}$
is
an orthogonal projection, and $0<b_{\Omega}<\rho_0$;
hence the norm of the derivative of $\Pi_{\Omega}$ is at most $C_{\Omega}$.
Consequently, on every interval contained in $U$ the function $w(\cdot)$
is
Lipschitz with constant $C_{\Omega}L$, and the second
line of \eqref{eq:derivativeofprojection} holds at every
$t\in U$ at which $z(\cdot)$ is differentiable. Now let $\alpha\leq s<t\leq\beta$.
If $[s,t]\subset U$, then $|w(t)-w(s)|\leq C_{\Omega}L(t-s)$.
Otherwise, let $s'$ and $t'$ be the smallest and the largest element of
$C\cap[s,t]$. Since $w(\cdot)$ is continuous,
$[s,s')\subset U$, $(t',t]\subset U$ and $w(\cdot)=z(\cdot)$ on $C$,
\begin{align*}
|w(t)-w(s)| \leq & \ |w(t)-w(t')| + |z(t')-z(s')| + |w(s')-w(s)|
\\
\leq & \ C_{\Omega}L(t-t') + L(t'-s') + C_{\Omega}L(s'-s) \leq C_{\Omega}L(t-s).
\end{align*}
Finally, the function $w(\cdot)-z(\cdot)$ is Lipschitz and vanishes on $C$.
Therefore its derivative is zero at every accumulation
point of $C$ at which it is differentiable; since the points of $C$ that are
not accumulation points of $C$ form a countable set,
$\dot{w}(t)=\dot{z}(t)$ for a.e.~$t\in C$.
\end{proof}
\begin{lemma}\label{lemma:approximation}
Let $\{x_n\}_{n=1}^{\infty}\subset\bar{\Omega}$ be a sequence that converges to $x_0\in\bar{\Omega}$, let $\tilde{u}(\cdot)\in\Gamma_2^{K,x_0}$
and $\tilde{x}(\cdot):=x^{\tilde{u}}_{x_0}(\cdot)$. Then there exist $n_0\in\N$ and controls $\tilde{u}^n(\cdot)\in A\Gamma_2^{x_n}$, $n\geq n_0$, such that
\begin{enumerate}[(i)]
\item $|\tilde{u}^n(t)|\leq C_{\Omega}K_1$ for a.e.~$t\in[0,T]$ and all $n\geq n_0$;
\item the paths $x^{\tilde{u}^n}_{x_n}(\cdot)$ converge to $\tilde{x}(\cdot)$ uniformly on $[0,T]$;
\item $\tilde{u}^n(t)\to\tilde{u}(t)$ for a.e.~$t\in[0,T]$.
\end{enumerate}
\end{lemma}
\begin{proof}
Let $a_n:=x_n-x_0$ and let $n_0$ be such that $|a_n|<\rho_0$ for $n\geq n_0$.
For $n\geq n_0$ let $z_n(t):=\tilde{x}(t)+a_n$.
Since $\tilde{x}(t)\in\bar{\Omega}$, we have $d_{\Omega}(z_n(t))\leq|a_n|<\rho_0$; moreover $z_n(\cdot)$ is continuously differentiable,
$\dot{z}_n(\cdot)=\tilde{u}(\cdot)$ and $|\tilde{u}(\cdot)|\leq K_1$. Let
\begin{align*}
w_n(t):=\Pi_{\Omega}(z_n(t)), \qquad \tilde{u}^n(t):=\dot{w}_n(t).
\end{align*}
By Lemma \ref{lemma:projection}, $w_n(\cdot)$ takes values in $\bar{\Omega}$
and is Lipschitz with constant $C_{\Omega}K_1$; also, $w_n(0)=\Pi_{\Omega}(x_n)=x_n$.
Hence $\tilde{u}^n(\cdot)\in A\Gamma_2^{x_n}$, $x^{\tilde{u}^n}_{x_n}(\cdot)=w_n(\cdot)$, and (i) holds. Since
$|w_n(t)-\tilde{x}(t)|\leq|w_n(t)-z_n(t)|+|a_n|\leq 2|a_n|$, (ii) holds.

To prove (iii), let $t\in(0,T)$ be such that \eqref{eq:derivativeofprojection}
holds at $t$ for $w_n(\cdot)$ and $z_n(\cdot)$, for every $n\geq n_0$;
almost every $t$ has this property. If $\tilde{x}(t)\in\Omega$, then $z_n(t)\in\Omega$ for all large $n$, and for such $n$ we have $\tilde{u}^n(t)=\tilde{u}(t)$.
Suppose that $\tilde{x}(t)\in\partial\Omega$. The function $r\mapsto b_{\Omega}(\tilde{x}(r))$ is
continuously differentiable in a neighbourhood of $t$,
it is nonpositive because $\tilde{x}(\cdot)$ takes values in $\bar{\Omega}$,
and it vanishes at $t$. Thus it has a local maximum at $t$, so that
\begin{align*}
Db_{\Omega}(\tilde{x}(t))\cdot\tilde{u}(t)=0.
\end{align*}
Let $n\geq n_0$. If $z_n(t)\in\bar{\Omega}$, then $\tilde{u}^n(t)=\tilde{u}(t)$.
If $z_n(t)\notin\bar{\Omega}$, then, by \eqref{eq:derivativeofprojection},
\begin{align*}
|\tilde{u}^n(t)-\tilde{u}(t)| \leq & \ \big|\big(Db_{\Omega}(z_n(t))-Db_{\Omega}(\tilde{x}(t))\big)\cdot\tilde{u}(t)\big| + b_{\Omega}(z_n(t))\|D^2b_{\Omega}\|_{\infty}K_1
\leq 2\|D^2b_{\Omega}\|_{\infty}K_1|a_n|,
\end{align*}
because the segment from $\tilde{x}(t)$ to $z_n(t)$ is contained in
$\Sigma_{\rho_0}$ and $b_{\Omega}(z_n(t))=d_{\Omega}(z_n(t))\leq|a_n|$.
In all cases, $\tilde{u}^n(t)\to\tilde{u}(t)$.
\end{proof}
\begin{lemma}\label{lemma:Gammacetaisclosed} 
The set-valued mapping 
\begin{align*}
\Gamma^{\cdot}[\cdot]:
\mathscr{P}_{m_0}(\Gamma^K)\times \bar{\Omega} &\ \longsquiggly \Gamma_2^K
\\
(\eta,x_0) &\ \longmapsto \Gamma^{\eta}[x_0]
\end{align*}
is closed.
\end{lemma}
\begin{proof}
 Let $\{x_n\}_{n=1}^{\infty}\subset \bar{\Omega}$ be a
 sequence
 that converges to $x_0\in \bar{\Omega}$,
 $\{\eta^n\}_{n=1}^{\infty} \subset \mathscr{P}_{m_0}(\Gamma^K)$
 a sequence converging to $\eta\in\mathscr{P}_{m_0}(\Gamma^K)$,
 and let $\{u_n(\cdot)\}_{n=1}^{\infty}\subset \Gamma_2^K$
 be a sequence with $u_n(\cdot) \in \Gamma^{\eta^n}[x_n]$ for every $n$,
 converging in $d_2$ (i.e.~uniformly) to $u(\cdot)\in\Gamma_2^{K}$.
 Since
\begin{align*}
\max_{0\leq t\leq T}|x^{u_n}_{x_n}(t) - x^{u}_{x_0}(t)| \leq |x_n-x_0|
+ T d_2(u_n,u),
\end{align*}
 the sequence of paths $\{x^{u_n}_{x_n}(\cdot)\}_{n=1}^{\infty}$
 converges uniformly to the path $x_{x_0}^{u}(\cdot)$, which therefore
 takes values in $\bar{\Omega}$,
 so $u(\cdot) \in \Gamma_2^{K,x_0}$.
Let $\tilde{u}(\cdot)\in \Gamma_2^{K,x_0}$ be arbitrary, let
$\tilde{u}^n(\cdot)\in A\Gamma_2^{x_n}$, $n\geq n_0$, be the controls
given by Lemma \ref{lemma:approximation}, and let $\tilde{x}^n(\cdot):=x^{\tilde{u}^n}_{x_n}(\cdot)$.
Since the evaluation mappings $e_t:\Gamma^K\to M$ are $1$-Lipschitz,
\begin{align*}
\mathbf{d}_{M}((e_t)_{\#}\eta,(e_t)_{\#}\eta^n)\leq \mathbf{d}_{\Gamma^K}(\eta,\eta^n)
\end{align*}
for all $n$ and $t\in[0,T]$. By this inequality, (L-\ref{L-ii})(b), the continuity of $l$ and Lemma \ref{lemma:approximation},
\begin{align*}
l(t,\tilde{x}^n(t),\tilde{u}^n(t),(e_t)_{\#}\eta^n) \ \longrightarrow \ l(t,x^{\tilde{u}}_{x_0}(t),\tilde{u}(t),(e_t)_{\#}\eta)
\qquad \textrm{for a.e. } t\in[0,T],
\end{align*}
and these functions are bounded uniformly in $n$ and $t$, because $|l|\leq C_0(1+C_{\Omega}^2K_1^2)$
on $[0,T]\times\bar{\Omega}\times\bar{B}_{C_{\Omega}K_1}(0)\times\mathscr{P}(M)$ by \eqref{eq:quadraticbounds}.
Similarly, $l_T(\tilde{x}^n(T),(e_T)_{\#}\eta^n)\to l_T(x^{\tilde{u}}_{x_0}(T),(e_T)_{\#}\eta)$, by (T-\ref{T-ii}) and the continuity of $l_T$ in $x$.
By the dominated convergence theorem,
\begin{equation}\label{eq:nowthisisimmediate}
I^{x_0}[\tilde{u};\eta] = \lim_{n\to\infty}I^{x_n}[\tilde{u}^n;\eta^n].
\end{equation}
Likewise, the uniform convergence of $u_n(\cdot)$ and $x^{u_n}_{x_n}(\cdot)$ gives
\begin{equation}\label{eq:thistoo}
I^{x_0}[u;\eta] = \lim_{n\to\infty} I^{x_n}[u_n;\eta^n].
\end{equation}
Since $\tilde{u}^n(\cdot)\in A\Gamma_2^{x_n}$ and $u_n(\cdot)\in \Gamma^{\eta^n}[x_n]$,
by Lemma \ref{lemma:optimaloverAC} we have
\begin{align*}
I^{x_n}[u_n;\eta^n]\leq I^{x_n}[\tilde{u}^n;\eta^n], \quad n\geq n_0.
\end{align*}
Passing to the limit as $n\to\infty$, by (\ref{eq:nowthisisimmediate}) and (\ref{eq:thistoo}),
we obtain
\begin{align*}
I^{x_0}[u(\cdot);\eta] \leq I^{x_0}[\tilde{u};\eta].
\end{align*}
Since $\tilde{u}(\cdot)$ was an arbitrary element of $\Gamma_2^{K,x_0}$,
this has shown that $u(\cdot)\in \Gamma^{\eta}[x_0]$,  and,
thus, that the mapping $\Gamma^{\cdot}[\cdot]$ is closed.
\end{proof}
Consequently, a measurable selection $s$ exists. The mapping $x_0\mapsto(x^{s(x_0)}_{x_0},s(x_0))$ is then Borel, being
the composition of $x_0\mapsto(x_0,s(x_0))$ with the continuous mapping
$(x_0,u(\cdot))\mapsto(x^u_{x_0}(\cdot),u(\cdot))$
from $\{(x_0,u(\cdot)) \ | \ x_0\in\bar{\Omega}, u(\cdot)\in\Gamma_2^{K,x_0}\}$
to $\Gamma^K$. Hence the measure $\eta'$ above is well defined,
it belongs to $\wassspaceonGammamKzero$, and $\eta'(\mathcal{G}(\eta))=1$, because
$(x^{s(x_0)}_{x_0},s(x_0))\in \mathcal{G}(\eta)$ for every $x_0\in\bar{\Omega}$. Thus, $E(\eta)\neq\emptyset$.
\subsection{Existence of equilibrium and mild solution}

We cite: 

\textbf{Kakutani-Fan-Glicksberg theorem.} \cite[17.55]{aliprantisborder} 
\emph{Let $S$ be a nonempty compact convex subset of a locally convex Hausdorff space,
and let the multi-valued mapping $\phi: S\to S$ have closed graph and
nonempty convex values. 
Then the set of fixed points of $\phi$ is compact and nonempty.}

As mentioned in Section \ref{subsection:notationandhypotheses}, $\mathscr{P}(\Gamma^K)$ is a 
compact metric space with the 1-Wasserstein distance, since $\Gamma^K$ 
is compact (w.r.t.~uniform convergence). Because $\Gamma^K$ is compact, 
it is known that narrow convergence of sequences in $\mathscr{P}(\Gamma^K)$ 
coincides with convergence in the Wasserstein distance \cite[Proposition 7.1.5]{gradientflows}. 
Every element $\mu \in \mathscr{P}(\Gamma^K)$ can be identified with the 
linear functional $f\mapsto \int_{\Gamma^K} f(\gamma) \mu(d\gamma)$ on $C(\Gamma^K)$. 
Narrow convergence is induced by the weak-$\star$ topology 
of $C(\Gamma^K)'$ ---the continuous dual of the space of 
continuous functions on $\Gamma^K$, and $C(\Gamma^K)'$, with its weak-$\star$ topology, 
is a locally convex, Hausdorff, topological vector space. Thus, $\mathscr{P}(\Gamma^K)$ 
is a compact, convex subset of $C(\Gamma^K)'$, and so
is $\mathscr{P}_{m_0}(\Gamma^K)$, because it is convex, and it is closed,
because
the mapping $\eta\mapsto(\pi^1\circ e_0)_{\#}\eta$ is continuous, and it is nonempty because it contains the image of $m_0$
under the mapping that assigns to $x_0\in\bar{\Omega}$ the pair $(x(\cdot),u(\cdot))\in\Gamma^K$ with $x(\cdot)\equiv x_0$, $u(\cdot)\equiv 0$. Having seen that the multivalued mapping 
$E:\mathscr{P}_{m_0}(\Gamma^K) \longsquiggly  \mathscr{P}_{m_0}(\Gamma^K)$ 
takes on nonempty values, now we will check that they are convex and that the 
mapping has a closed graph. 
\begin{lemma}\label{lemma:closedandconvexvalues}
The set-valued mapping $E:\mathscr{P}_{m_0}(\Gamma^K) \longsquiggly  \mathscr{P}_{m_0}(\Gamma^K)$, 
defined above in (\ref{eq:defnofE}), takes on non-empty, convex values and has a closed graph.
\end{lemma}
\begin{proof}
The nonemptiness has been shown. If $\eta \in \mathscr{P}_{m_0}(\Gamma^K)$,
$0 <\lambda < 1$ and $\hat\eta^0$ and $\hat\eta^1$ are elements of
$E(\eta)$, then $(\lambda \hat\eta^0 + (1-\lambda)\hat\eta^1)(\mathcal{G}(\eta)) = \lambda + (1-\lambda) = 1$.
This shows that $E(\eta)$ is convex.

Let $\{\eta^j\}_{j=1}^{\infty}\subset \mathscr{P}_{m_0}(\Gamma^K)$
be a
sequence converging to $\eta\in \mathscr{P}_{m_0}(\Gamma^K)$,
take a sequence $\{\hat\eta^j\}_{j=1}^{\infty} \subset
\mathscr{P}_{m_0}(\Gamma^K)$ with $\hat\eta^j\in E(\eta^j)$,
and suppose $\lim_{j\to\infty}\mathbf{d}_{\Gamma^K}(\hat\eta,\hat\eta^j) = 0$
for some $\hat\eta\in\mathscr{P}_{m_0}(\Gamma^K)$.
For each $j$, the set $\mathcal{G}(\eta^j)$ is closed and $\hat\eta^j(\mathcal{G}(\eta^j))=1$, so that
\begin{equation}\label{eq:sptinF}
\spt(\hat\eta^j)\subset \mathcal{G}(\eta^j).
\end{equation}
Let $(\hat{x},\hat{u})\in \spt(\hat{\eta})$ be arbitrary.
We claim that there exists a sequence $\{(\hat{x}^j,\hat{u}^j)\}_{j=1}^{\infty}\subset \Gamma^K$,
with $(\hat{x}^j,\hat{u}^j)\in\spt(\hat{\eta}^j)$ for all $j$,
such that $(\hat{x}^j,\hat{u}^j)\overset{d_{\Gamma}}{\to} (\hat{x},\hat{u})$ (cf.~\cite[Prop.~5.1.8]{gradientflows}).
Indeed, for $r>0$ let $B_r\subset\Gamma^K$ be the open ball of radius $r$
centered at $(\hat{x},\hat{u})$. By the definition of support,
$\hat{\eta}(B_r)>0$, and, since $\hat{\eta}^j$ converges narrowly to
$\hat{\eta}$ and $B_r$ is open, the Portmanteau theorem gives
$\liminf_{j\to\infty}\hat{\eta}^j(B_r)\geq\hat{\eta}(B_r)>0$.
Thus, for every $r>0$, the set $\spt(\hat{\eta}^j)$ intersects $B_r$
for all sufficiently large $j$, which proves the claim.
Let $x_0=\hat{x}(0)$ and $x_j = \hat{x}^j(0)$,  $j=1,2,\ldots$ By \eqref{eq:sptinF},
$\hat{u}^j\in \Gamma^{\eta^j}[x_j]$ for each $j$.
In summary, we have sequences $\eta^j\to\eta$,
$x_j\to x_0$ and $\{\hat{u}^j\}_{j=1}^{\infty}\subset \Gamma_2^K$ such that
$\hat{u}^j\in \Gamma^{\eta^j}[x_j]$ for every $j$ and the
controls $\hat{u}^j$ converge to $\hat{u}$ in $d_2$.
We are exactly in the situation of
Lemma \ref{lemma:Gammacetaisclosed},
so we conclude that $\hat{u}\in\Gamma^{\eta}[x_0]$, i.e.~$(\hat{x},\hat{u})\in \mathcal{G}(\eta)$. Thus, we
have shown that
\begin{align*}
\spt(\hat{\eta}) \subset \mathcal{G}(\eta),
\end{align*}
which means that $\hat{\eta}(\mathcal{G}(\eta))=1$, i.e.~$\hat{\eta}\in E(\eta)$.
So we have proved that the set-valued mapping $E$ has a closed graph.
\end{proof}
Therefore, by the Kakutani-Fan-Glicksberg theorem, the set-valued 
mapping $E:\mathscr{P}_{m_0}(\Gamma^K)$ $\longsquiggly  \mathscr{P}_{m_0}(\Gamma^K)$ 
has a fixed point. By Remark \ref{remark:verifyequivalence}, this means that there
 exists a constrained equilibrium of mean field game of controls, i.e.~we have proved:
\begin{lemma}\label{lemma:existenceofCEMFGC}
With the notation of section \ref{subsection:notationandhypotheses} let the 
hypotheses (L-\ref{L-i}) through (L-\ref{L-iii}) and (T-\ref{T-i}), (T-\ref{T-ii}) 
hold, and suppose that $c_1\leq\bar{c}$, where $\bar{c}$ and $K=(K_1,K_2)$ are
as in Corollary \ref{cor:ifmonesmallenough}. Then there exists a \emph{CEMFGC} 
(constrained equilibrium of mean field game of controls) (see Definition \ref{defn:CEMFGC}) 
$\eta$ in $\mathscr{P}_{m_0}(\Gamma^K)$. 
\end{lemma}
In the same way as \cite{cannarsacapuani1}, we define the notion of mild 
solution for this problem. Recall that $X$ is an open set including 
$\bar{\Omega}$ and $M=\bar{\Omega}\times\R^d$.
\begin{defn}\label{defn:mildsolution}
Let $l:[0,T]\times X\times \R^d\times\mathscr{P}(M)\to\R$,  
$l_T:X\times\mathscr{P}(M)\to \R$ be continuous functions, 
$m_0\in\mathscr{P}(\bar{\Omega})$.  A 
pair $(V,\sigma)\in C([0,T]\times\bar{\Omega})\times C([0,T];\mathscr{P}(\bar{\Omega}))$ is 
called a \emph{mild solution of the constrained MFGCs problem with initial state distribution $m_0$} 
if there exists a CEMFGC $\eta\in \mathscr{P}_{m_0}(\Gamma^K)$, with $K=(K_1,K_2)$ as in
Corollary \ref{cor:ifmonesmallenough}, such that
\begin{enumerate}[(i)]
\item $\sigma_0 = m_0$,
\item $\sigma_t = (\pi^1\circ e_t)_{\#}\eta$ for all $t\in [0,T]$,  
\item \label{item:thmiii} for every $x\in\bar{\Omega}$ and $t\in [0,T]$, with $\nu_s:=(e_s)_{\#}\eta$,
\begin{align*}
V(t,x) = \inf\Big\{ \int_t^T l(s,\gamma(s),\dot{\gamma}(s),\nu_s) ds + l_T(\gamma(T),\nu_T)
\ \Big| \ \gamma(\cdot)\in AC([t,T];\bar{\Omega}), \ \gamma(t)=x \Big\}.
\end{align*} 
\end{enumerate}
\end{defn}
Thus, the main theorem follows:
\begin{thm}[Existence of mild solution]\label{thm:mildsolution}
With the notation of section \ref{subsection:notationandhypotheses} let the 
hypotheses (L-\ref{L-i}) through (L-\ref{L-iii}) and (T-\ref{T-i}), (T-\ref{T-ii}) 
hold, and suppose that $c_1\leq\bar{c}$, with $\bar{c}$ as in
Corollary \ref{cor:ifmonesmallenough}. Let $m_0\in\mathscr{P}(\bar{\Omega})$. 
Then there exists a mild solution 
$(V,\sigma)\in C([0,T]\times\bar{\Omega})\times C([0,T];\mathscr{P}(\bar{\Omega}))$ 
of the constrained MFGCs problem with initial state distribution $m_0$.
\end{thm}
\begin{proof}
We apply Lemma \ref{lemma:existenceofCEMFGC} to obtain a CEMFGC
$\eta\in\mathscr{P}_{m_0}(\Gamma^K)$, with $K=(K_1,K_2)$ as in Corollary \ref{cor:ifmonesmallenough}.
Let $\sigma_t = (\pi^1\circ e_t)_{\#}\eta$ for all $t\in[0,T]$. Then $\sigma_0=m_0$, and
$\sigma_{\cdot}$ is continuous; in fact, $\mathbf{d}_{\bar{\Omega}}(\sigma_s,\sigma_t)\leq K_1|t-s|$,
because $|x(t)-x(s)|\leq K_1|t-s|$ for every $(x,u)\in\Gamma^K$. Then, define $V(t,x)$
as in (\ref{item:thmiii}).
As in Remark \ref{remark:Ltol}, let $\nu_s = (e_s)_{\#}\eta$,  $0\leq s\leq T$,
$f(s,\xi,v) := l(s,\xi,v,\nu_s)$ and $g(\xi):=l_T(\xi,\nu_T)$. By \eqref{eq:quadraticbounds},
\begin{equation}\label{eq:quadraticboundsf}
\frac{1}{4c}|v|^2 - C_0 \leq f(s,\xi,v) \leq C_0(1+|v|^2),
\end{equation}
and, since $c_1\textrm{Lip}(\nu_{\cdot})\leq \bar{c}(K_1+K_2)=1$ by \eqref{eq:LipnuK}, inequality \eqref{eq:timeregf} gives
\begin{equation}\label{eq:timeregfone}
|f(s,\xi,v)-f(s',\xi,v)| \leq (\kappa_1+1)(1+|v|^2)|s-s'|.
\end{equation}
The value $V(t,x)$ is finite for all $(t,x)\in [0,T]\times\bar{\Omega}$, because
the cost of the constant path is at most $n_1T+\|l_T\|_{\infty}$ by (L-\ref{L-i}), and the cost of any admissible path is at least $-C_0T-\|l_T\|_{\infty}$ by
\eqref{eq:quadraticboundsf}. Moreover, if $t<T$, then a minimizing path for the problem defining $V(t,x)$ exists, for
the same reasons as in the case $t=0$ (see the paragraph that follows \eqref{eq:ACprob}); and $V(T,x)=g(x)$.
It remains to show the continuity of $V$.
Let $t<T$ and let $\gamma(\cdot)$ be a minimizing path for $V(t,x)$.
Comparing its cost with the cost of the constant path,
from \eqref{eq:quadraticboundsf} we get
\begin{equation}\label{eq:energybound}
\int_t^T|\dot{\gamma}(s)|^2ds \leq R := 4c\big[(n_1+C_0)T + 2\|l_T\|_{\infty}\big],
\quad \textrm{hence} \quad |\gamma(s)-\gamma(s')|\leq (R|s-s'|)^{1/2}, \ \ s,s'\in[t,T].
\end{equation}
Let $\omega_g$ be a nondecreasing modulus of continuity of $g$ on $\bar{\Omega}$.

Let us address the continuity in $x$. Let $t\in[0,T)$ and $x,x'\in\bar{\Omega}$ with $\tau:=|x-x'|\in(0,\rho_0)$,
and let $\gamma(\cdot)$ be a minimizing path for $V(t,x)$.
For $r\in[0,1]$ let $z_r:=x'+r(x-x')$. Since $d_{\Omega}(z_r)\leq|z_r-x|\leq\tau<\rho_0$, we can define
$\zeta(r):=\Pi_{\Omega}(z_r)$, where $\Pi_{\Omega}$ is the projection of
Lemma \ref{lemma:projection}. By that lemma,
$\zeta:[0,1]\to\bar{\Omega}$ is Lipschitz with constant $C_{\Omega}\tau$;
moreover $\zeta(0)=x'$ and $\zeta(1)=x$.
If $T-t\geq\tau$, consider the path $\gamma'(\cdot)\in AC([t,T];\bar{\Omega})$
given by
$\gamma'(s) := \zeta((s-t)/\tau)$ for $s\in[t,t+\tau]$ and $\gamma'(s):=\gamma(s-\tau)$ for $s\in[t+\tau,T]$.
Since $\gamma'(t)=x'$, by \eqref{eq:quadraticboundsf}, \eqref{eq:timeregfone} and \eqref{eq:energybound},
\begin{align*}
V(t,x') \leq & \ C_0(1+C_{\Omega}^2)\tau + \int_t^{T-\tau} f(s+\tau,\gamma(s),\dot{\gamma}(s))ds + g(\gamma(T-\tau))
\\
\leq & \ C_0(1+C_{\Omega}^2)\tau + \int_t^{T-\tau} f(s,\gamma(s),\dot{\gamma}(s))ds + (\kappa_1+1)(T+R)\tau
+ g(\gamma(T)) + \omega_g((R\tau)^{1/2})
\\
\leq & \ V(t,x) + \big[C_0(2+C_{\Omega}^2) + (\kappa_1+1)(T+R)\big]\tau + \omega_g((R\tau)^{1/2}),
\end{align*}
where in the last line we used that $\int_{T-\tau}^T f(s,\gamma(s),\dot{\gamma}(s))ds\geq -C_0\tau$.
If $T-t<\tau$, comparing with the constant path we have $V(t,x')\leq n_1\tau + g(x') \leq n_1\tau + g(x) + \omega_g(\tau)$,
while $V(t,x)\geq -C_0\tau + g(\gamma(T)) \geq -C_0\tau + g(x) - \omega_g((R\tau)^{1/2})$.
In both cases, exchanging the roles of $x$ and $x'$, and recalling that $V(T,\cdot)=g(\cdot)$, we obtain, for all $t\in[0,T]$,
\begin{equation}\label{eq:Vcontinuousinx}
|V(t,x)-V(t,x')|\leq \omega(|x-x'|), \quad
\omega(\tau) := \big[C_0(2+C_{\Omega}^2) + n_1 + (\kappa_1+1)(T+R)\big]\tau + \omega_g(\tau) + \omega_g((R\tau)^{1/2}).
\end{equation}

Now we address the continuity in $t$. Let $0\leq t<t'\leq T$ with $(R(t'-t))^{1/2}<\rho_0$,
$x\in\bar{\Omega}$, and let $\gamma(\cdot)$ be a
minimizing path for $V(t,x)$. Following the constant path on $[t,t']$ and then a
minimizing path for $V(t',x)$
we get $V(t,x)\leq n_1(t'-t) + V(t',x)$. On the other hand, since the restriction
of $\gamma(\cdot)$ to $[t',T]$
is admissible for the problem defining $V(t',\gamma(t'))$, by \eqref{eq:quadraticboundsf}, \eqref{eq:energybound}
and \eqref{eq:Vcontinuousinx},
\begin{align*}
V(t,x) \geq -C_0(t'-t) + V(t',\gamma(t')) \geq -C_0(t'-t) + V(t',x) - \omega((R(t'-t))^{1/2}).
\end{align*}
These two inequalities and \eqref{eq:Vcontinuousinx} show that $V$ is continuous
on $[0,T]\times\bar{\Omega}$.
\end{proof}
The function $V$ of Definition \ref{defn:mildsolution} is 
referred to as the \textit{value function} of the constrained mean field game.

\subsection{Uniqueness (under monotonicity) of the value function component}
The monotonicity condition on the Lagrangian $l$ and the terminal cost 
$l_T$ below is simply an elementary extension of the one used in \cite{cannarsacapuani1},
while the proof of Theorem \ref{thm:uniqueness} also mirrors 
the one of its analogue in \cite{cannarsacapuani1}; for the sake of completeness,
we nevertheless provide it.

Recall that $M = \bar{\Omega}\times \R^d$ and that the costs depend on the joint 
distribution of states and controls.  
\begin{defn}
\label{def:monotone}
For our Lagrangian $l$,
\begin{itemize}
  \item we say that $l$ is \emph{monotone} if for every $t\in[0,T]$ and all
    $\nu_1,\nu_2\in\mathscr{P}(M)$ with bounded support,
    \begin{equation}\label{eq:monotone-l}
      \int_{M}\bigl(l(t,x,v,\nu_1)-l(t,x,v,\nu_2)\bigr)(\nu_1-\nu_2)(dxdv)\;\geq\;0
    \end{equation}
    and we say that $l$ is \emph{strictly monotone} if, in addition, equality in 
    \eqref{eq:monotone-l} implies 
    \[
      l(t,x,v,\nu_1)=l(t,x,v,\nu_2)\qquad\forall (x,v)\in M;
    \]
  \item similarly, $l_T$ is \emph{monotone} if for all $\nu_1,\nu_2\in\mathscr{P}(M)$ with bounded support,
    \begin{equation}\label{eq:monotone-lT}
      \int_{M}\bigl(l_T(x,\nu_1)-l_T(x,\nu_2)\bigr)(\nu_1-\nu_2)(dxdv)\;\geq\;0
    \end{equation}
    and \emph{strictly monotone} if equality forces 
    $l_T(x,\nu_1)=l_T(x,\nu_2)$ for all $x\in\bar{\Omega}$.
\end{itemize}
\end{defn}
  We consider measures with bounded support so that the integrals in \eqref{eq:monotone-l} are finite
  (recall that $l$ has quadratic growth in $v$); the measures $(e_t)_{\#}\eta$ with $\eta\in\mathscr{P}(\Gamma^K)$
  are supported in $\bar{\Omega}\times\bar{B}_{K_1}(0)$.
  In \eqref{eq:monotone-lT} the integrand does not depend on $v$, but writing the 
  integral over $M$ (instead of over $\bar{\Omega}$) will simplify the notation 
  in the proof below. Also, to free up the subindex from the notation 
  of paths of measures for this short final discussion,
  let us denote the dependence on the time
  parameter by including it in a round bracket: $\nu(t)=\nu_t$.
\begin{thm}[Uniqueness of the value function of the mild solution]
\label{thm:uniqueness}
 Under the same hypotheses of Theorem \ref{thm:mildsolution} assume,
 additionally, that $l$ and $l_T$ are strictly monotone in the sense of 
  Definition~\ref{def:monotone}.  
  Let $\eta_1,\eta_2\in\mathscr{P}_{m_0}(\Gamma^K)$ be two constrained equilibria 
  (CEMFGC) and let $V_1,V_2$ be the associated value functions defined by
  \[
    V_i(t,x)=\inf\Bigl\{
      \int_t^T l\bigl(s,\gamma(s),\dot{\gamma}(s),\nu_i(s)\bigr)ds
      +l_T\bigl(\gamma(T),\nu_i(T)\bigr) \ \Big| \ \gamma(\cdot)\in AC([t,T];\bar{\Omega}), \ \gamma(t)=x \Bigr\},
  \]
  where $\nu_i(s)=(e_s)_{\#}\eta_i$ for $s\in[0,T]$.  
  Then $V_1=V_2$ on $[0,T]\times\bar{\Omega}$.  
  Consequently, if $(V_1,\sigma_1)$ and $(V_2,\sigma_2)$ are two mild solutions of 
  the constrained MFGCs problem (Definition~\ref{defn:mildsolution}), then 
  $V_1=V_2$.
\end{thm}
\begin{proof}
  Let $\eta_1,\eta_2$ be as in the statement, $\nu_i(t)=(e_t)_{\#}\eta_i$ for every
  $t$, \(i=1,2\).
  Let $\gamma=(x,u)\in \mathcal{G}(\eta_1)$, which is a set of full $\eta_1$-measure,
  and let $x_0:=x(0)$.
  Then $x=x^u_{x_0}$, $u=\dot{x}$ and, by Lemma \ref{lemma:optimaloverAC},
  $x(\cdot)$ is an optimal trajectory for the problem defining $V_1(0,x_0)$.
  Then
  \[
    V_1(0,x_0)=\int_0^T l(t,x(t),u(t),\nu_1(t))dt+l_T(x(T),\nu_1(T)).
  \]
  The trajectory $x(\cdot)$ is admissible for the problem defining $V_2(0,x_0)$,
  but not necessarily optimal, hence the inequality:
  {\small
  \[
    l_T(x(T),\nu_1(T))-l_T(x(T),\nu_2(T))
    \le V_1(0,x_0)-V_2(0,x_0)
      +\int_0^T\bigl[l(t,x(t),u(t),\nu_2(t))-l(t,x(t),u(t),\nu_1(t))\bigr]dt.
  \]
  }
  Now we integrate this inequality over $\Gamma^K$ with respect to $\eta_1$
  (all the terms are bounded and continuous functions of $\gamma\in\Gamma^K$; for $V_i(0,x(0))$
  see the proof of Theorem \ref{thm:mildsolution});
  because $\eta_1(\mathcal{G}(\eta_1))=1$ and
  $(\pi^1\!\circ e_0)_{\#}\eta_1=m_0$, we obtain
  \begin{align*}
    &\int_{\Gamma^K}\bigl[l_T(x(T),\nu_1(T))-l_T(x(T),\nu_2(T))\bigr]\eta_1(d\gamma)
    \\
    \le & \ \int_{\bar{\Omega}}\bigl[V_1(0,x)-V_2(0,x)\bigr]m_0(dx)
    +
    \int_{\Gamma^K}\!\int_0^T\bigl[l(t,x(t),u(t),\nu_2(t))
        -l(t,x(t),u(t),\nu_1(t))\bigr]dt\eta_1(d\gamma).
  \end{align*}
Since $\nu_1(t)=(e_t)_{\#}\eta_1$ for every
  $t$ and $e_t(\gamma)=(x(t),u(t))$,
  \[
    \int_{\Gamma^K} l_T(x(T),\nu_i(T))\eta_1(d\gamma)
    =\int_{M} l_T(x,\nu_i(T))\nu_1(T)(dxdv),
  \]
  and, by Fubini's theorem,
  \[
    \int_{\Gamma^K}\!\int_0^T l(t,x(t),u(t),\nu_i(t))dt\;\eta_1(d\gamma)
    =\int_0^T\!\int_{M} l(t,x,v,\nu_i(t))\nu_1(t)(dxdv)dt .
  \]
  Therefore the integrated inequality becomes
  \begin{equation}\label{eq:int1}
    \begin{aligned}
      &\int_{M}\bigl[l_T(x,\nu_1(T))-l_T(x,\nu_2(T))\bigr]\nu_1(T)(dxdv)
      \\
      &\le \int_{\bar{\Omega}}\bigl[V_1(0,x)-V_2(0,x)\bigr]m_0(dx
      +\int_0^T\!\int_{M}\bigl[l(t,x,v,\nu_2(t))-l(t,x,v,\nu_1(t))\bigr]\nu_1(t)(dxdv)dt .
    \end{aligned}
  \end{equation}
  Starting from the symmetric situation (optimal trajectory for $\eta_2$) and 
  integrating over $\eta_2$, we obtain the analogous inequality
  \begin{equation}\label{eq:int2}
    \begin{aligned}
      &\int_{M}\bigl[l_T(x,\nu_2(T))-l_T(x,\nu_1(T))\bigr]\nu_2(T)(dxdv)\\
      \le & \ \int_{\bar{\Omega}}\bigl[V_2(0,x)-V_1(0,x)\bigr]m_0(dx)
       +\int_0^T
       \!\int_{M}\bigl[l(t,x,v,\nu_1(t))-l(t,x,v,\nu_2(t))\bigr]\nu_2(t)(dxdv)dt.
    \end{aligned}
  \end{equation}
  Adding \eqref{eq:int1} and \eqref{eq:int2} cancels the terms involving the value 
  functions at time $0$, because $m_0$ is the same for both equilibria.  We get
  \begin{align*}
    &\int_{M}\bigl[l_T(x,\nu_1(T))-l_T(x,\nu_2(T))\bigr](\nu_1(T)-\nu_2(T))(dxdv)
    \\
    &\le \int_0^T\!\int_{M}\bigl[l(t,x,v,\nu_2(t))-l(t,x,v,\nu_1(t))\bigr]
        (\nu_1(t)-\nu_2(t))(dxdv)dt .
  \end{align*}
  Multiplying by $-1$ and rearranging,
  \begin{equation}\label{eq:sum}
    \begin{aligned}
       0 \geq & \ \int_0^T\!\int_{M}\bigl[l(t,x,v,\nu_1(t))-l(t,x,v,\nu_2(t))\bigr]
        (\nu_1(t)-\nu_2(t))(dxdv)dt
        \\
      & \ +   \int_{M}\bigl[l_T(x,\nu_1(T))-l_T(x,\nu_2(T))\bigr]
        (\nu_1(T)-\nu_2(T))(dxdv).
    \end{aligned}
  \end{equation}
  By the monotonicity assumption (Definition~\ref{def:monotone}), each integrand is 
  non‑negative for every $t$.  Hence the right‑hand side of \eqref{eq:sum} is a sum 
  (resp.~integral) of non‑negative terms that is non-positive.  Consequently, each term 
  must vanish identically:
  \begin{align}
    \int_{M}\bigl[l(t,x,v,\nu_1(t))-l(t,x,v,\nu_2(t))\bigr]
      (\nu_1(t)-\nu_2(t))(dxdv)&=0 
      \quad\text{for a.e.~}t\in[0,T],\label{eq:zerol}
      \\
    \int_{M}\bigl[l_T(x,\nu_1(T))-l_T(x,\nu_2(T))\bigr]
      (\nu_1(T)-\nu_2(T))(dxdv)&=0 .\label{eq:zerolT}
  \end{align}
  Strict monotonicity now yields
  \[
    l(t,x,v,\nu_1(t)) = l(t,x,v,\nu_2(t))\qquad \text{for all } (x,v)\in M,\;\text{a.e.~} t,
  \]
  and
  \[
    l_T(x,\nu_1(T)) = l_T(x,\nu_2(T))\qquad\forall x\in\bar{\Omega}.
  \]
  Therefore, for any admissible control $u(\cdot)$ and the corresponding trajectory 
  $x(\cdot)$, the running and terminal costs evaluated with $\nu_1$ and $\nu_2$ 
  coincide.  Hence, for every $(t,x)\in[0,T]\times\bar{\Omega}$, the functionals whose
  infima define
  $V_1(t,x)$ and $V_2(t,x)$ are identical, and $V_1=V_2$ on $[0,T]\times\bar{\Omega}$.

  Finally, if $(V_1,\sigma_1)$ and $(V_2,\sigma_2)$ are mild solutions, they arise 
  from some CEMFGC $\eta_1,\eta_2$, so the previous argument gives $V_1=V_2$.
\end{proof}

\begin{remark}
  Exactly as in \cite{cannarsacapuani1}, the strict monotonicity assumption does 
  not force the equilibrium measure $\eta$ itself to be unique, only the value 
  function. Uniqueness of the full mild solution (including the flow $\sigma$) 
  would require a stronger, separate condition. 
\end{remark}

\printbibliography

\end{document}